\documentclass{article}[12pt]
\usepackage{amsmath, amsfonts, amscd, amssymb, amsthm, enumerate, xypic, psfrag}

\DeclareMathOperator{\Mat}{\operatorname{M}}
\DeclareMathOperator{\GL}{\operatorname{GL}}
\DeclareMathOperator{\bfB}{\mathbf{B}}

\DeclareMathOperator{\card}{\#\,}
\DeclareMathOperator{\Ker}{\operatorname{Ker}}
\DeclareMathOperator{\Vect}{\operatorname{span}}
\DeclareMathOperator{\im}{\operatorname{Im}}
\DeclareMathOperator{\tr}{\operatorname{tr}}
\DeclareMathOperator{\rk}{\operatorname{rk}}
\DeclareMathOperator{\codim}{\operatorname{codim}}
\renewcommand{\setminus}{\smallsetminus}

\def\F{\mathbb{F}}
\def\K{\mathbb{K}}
\def\R{\mathbb{R}}

\def\calA{\mathcal{A}}
\def\calB{\mathcal{B}}

\def\calH{\mathcal{H}}

\def\calJ{\mathcal{J}}
\def\calK{\mathcal{K}}
\def\calL{\mathcal{L}}

\def\calP{\mathcal{P}}

\def\calR{\mathcal{R}}
\def\calS{\mathcal{S}}

\def\calV{\mathcal{V}}
\def\calW{\mathcal{W}}

\def\calY{\mathcal{Y}}

\def\lcro{\mathopen{[\![}}
\def\rcro{\mathclose{]\!]}}

\theoremstyle{definition}
\newtheorem{Def}{Definition}

\theoremstyle{plain}
\newtheorem{theo}{Theorem}
\newtheorem{prop}[theo]{Proposition}
\newtheorem{cor}[theo]{Corollary}
\newtheorem{lemme}[theo]{Lemma}

\theoremstyle{plain}
\newtheorem{conj}{Conjecture}

\theoremstyle{remark}
\newtheorem{Rems}{Remarks}
\newtheorem{Rem}[Rems]{Remark}

\title{The classification of large spaces of matrices with bounded rank}
\author{Cl\'ement de Seguins Pazzis\footnote{Lyc\'ee Priv\'e Sainte-Genevi\`eve, 2, rue
de l'\'Ecole des Postes, 78029 Versailles Cedex, FRANCE.}
\footnote{e-mail address: dsp.prof@gmail.com}}

\begin{document}

\thispagestyle{plain}

\maketitle

\begin{abstract}
Given an arbitrary (commutative) field $\K$, let $\calV$ be a linear subspace of $\Mat_n(\K)$ consisting of matrices of rank less than
or equal to some $r \in \lcro 1,n-1\rcro$. A theorem of Atkinson and Lloyd states that, if $\dim \calV>n\,r-r+1$ and $\card \K>r$, then either all the matrices of $\calV$ vanish everywhere on some common $(n-r)$-dimensional subspace of $\K^n$, or it is true of the matrices of the transposed space $\calV^T$.
Using a new approach, we prove that the restriction on the cardinality of the underlying field is unnecessary.
We also show that the results of Atkinson and Lloyd on the case
$\dim \calV=n\,r-r+1$ hold for any field, except in the special case when $n=3$, $r=2$ and $\K \simeq \F_2$.
In that exceptional situation, we classify all the exceptional spaces up to equivalence. Similar theorems of Beasley for rectangular matrices are
also extended to all fields.
Finally, we extend Atkinson, Lloyd and Beasley's classification theorems to a range of high dimensions which almost doubles
theirs, under the assumption that $\card \K>r$.
\end{abstract}



\section{Introduction}

\subsection{The problem}

Let $\K$ be a (commutative) field, and $n$ and $p$ be two positive integers, with $n \geq p$ unless specified otherwise. Denote by
$\Mat_{n,p}(\K)$ the space of all $n \times p$ matrices with entries in $\K$, and set
$\Mat_n(\K):=\Mat_{n,n}(\K)$. We denote by $\GL_n(\K)$ the group of invertible matrices of $\Mat_n(\K)$.

Given a positive integer $r$, an \textbf{$\overline{r}$-subspace} of $\Mat_{n,p}(\K)$
is a linear subspace in which all the matrices have rank less than or equal to $r$.
The study of $\overline{r}$-spaces of matrices has a long history dating back to Isaiah Schur, who was the first to discover
that a $\overline{1}$-subspace of matrices either consists of matrices whose images are included in a common $1$-dimensional space,
or consists of matrices whose kernels contain a common hyperplane. A naive generalization to larger values of $r$ fails:
of course, if a linear subspace $\calV$ of $\Mat_{n,p}(\K)$ is such that every matrix of $\calV$ vanishes everywhere on some
$(p-r)$-dimensional subspace of $\K^p$, or the column spaces of all the matrices of $\calV$ are included in a common $r$-dimensional subspace of $\K^n$,
then $\calV$ is an $\overline{r}$-space, but the converse does not hold.

One must be aware that the general classification problem for $\overline{r}$-spaces of matrices, i.e.,
the one of determining all the maximal $\overline{r}$-subspaces of $\Mat_{n,p}(\K)$ up to equivalence, is intractable.
Thus, one has to fall back to more modest aims.
So far, there have been two major kinds of classification results on $\overline{r}$-spaces.
For small values of $r$ (up to $r=3$), or when $n>1+\dfrac{r(r-1)}{2}$, Atkinson \cite{AtkinsonPrim} has obtained a classification of all the
\emph{primitive}\footnote{The notion of a primitive $\overline{r}$-space will play no
    part in our study, so we do not recall its definition.} $\overline{r}$-spaces if the underlying field $\K$
has more than $r$ elements. His results were later rediscovered by Eisenbud and Harris \cite{EisenbudHarris}, who used techniques
from algebraic geometry to reprove them in the special case of algebraically closed fields.

On the other hand, one can seek to understand the structure of maximal $\overline{r}$-subspaces of matrices with large dimension. That problem dates
back to Dieudonn\'e \cite{Dieudonne},
who proved that a linear subspace of $\Mat_n(\K)$ which contains only singular matrices - that is, an $\overline{n-1}$-subspace of $\Mat_n(\K)$ -
has dimension less than or equal to $n^2-n$, and that equality occurs only if
all the matrices of the subspace vanish at some common non-zero vector of $\K^n$ or
all the matrices have their column space included in a common hyperplane of $\K^n$.
This theorem was a major tool in Dieudonn\'e's determination of the automorphisms of the vector space $\Mat_n(\K)$
which preserve non-singularity \cite{Dieudonne}. Dieudonn\'e's result was later generalized by Flanders, who proved that an
$\overline{r}$-subspace of $\Mat_{n,p}(\K)$ must have dimension less than or equal to $n\,r$ (remember the assumption that $n \geq p$),
and classified the cases of equality
(note that equality is obviously attained by the space of all matrices with all last $p-r$ columns zero).
Unlike Dieudonn\'e, Flanders relied upon rudimentary techniques of algebraic geometry,
to the effect that his arguments only apply to fields with more than $r$ elements
(he also excluded fields of characteristic $2$, but that limitation can be easily worked around).
It was only much later \cite{Meshulam} that Roy Meshulam was eventually able to prove that Flanders's theorem holds for all fields. In the meantime, Atkinson and Lloyd had worked on extending the study of large $\overline{r}$-spaces to encompass
dimensions that are close to the maximal one $n\,r$. Among their results, they proved that if an $\overline{r}$-subspace $\calV$
of $\Mat_n(\K)$ has dimension greater than $nr-r+1$, then all the matrices of $\calV$ vanish everywhere on some common $(n-r)$-dimensional
subspace of $\K^n$ or their column spaces are included in a common $r$-dimensional subspace of $\Mat_n(\K)$.
It ensues that, up to equivalence and transposition, there is exactly one maximal $\overline{r}$-subspace of $\K^n$
of dimension greater than $nr-r+1$. Atkinson and Lloyd also classified the $\overline{r}$-spaces of dimension $nr-r+1$
and showed that, up to equivalence and transposition, there is exactly one such maximal $\overline{r}$-subspace of $\Mat_n(\K)$.
Later, Beasley \cite{Beasley} generalized Atkinson and Lloyd's results to spaces of rectangular matrices.

The works of Atkinson, Lloyd and Beasley are based upon Flanders's core ideas, and for this reason
they only address the case of fields with more than $r$ elements. In \cite{Meshulam}, Meshulam brought forth a counter-example which
cast doubt on the potential validity of Atkinson and Lloyd's results for fields with very small cardinality.

Thus, after the mid-1980's, there remained two open problems on the topic of large $\overline{r}$-spaces of square matrices:
\begin{enumerate}[(1)]
\item Obtain classification theorems for dimensions that are smaller than $nr-r+1$;
\item Search whether the Atkinson-Lloyd-Beasley theorems can be extended to all finite fields, minus some
exceptional cases.
\end{enumerate}
Until now, no progress had been made on either one of those problems. The main purpose of this article is to give a complete
solution to the second one. Our short answer is that the Atkinson-Lloyd-Beasley theorems
hold for all fields and all dimensions, with the notable exception of the case when $n=p=3$, $r=2$ and $\K \simeq \F_2$, in which we find
that, up to equivalence, Meshulam's counter-example is the sole exceptional solution.
In addition, we shall give a modest contribution to the first problem by covering a range of dimensions
that is roughly twice as large as the one in the Atkinson-Lloyd-Beasley theorems (though only for fields with more than $r$ elements).

Besides the sheer beauty of the Atkinson-Lloyd-Beasley classification theorems,
a strong motivation for generalizing them to all fields was their potential application to wide generalizations of Dieudonn\'e's
theorem on invertibility preservers. The following three theorems were proved in \cite{largedimpres} as an application of results of the present article:
to understand them, remember that a \textbf{Frobenius automorphism} is a map of the form
$$M \mapsto PMQ \quad \text{or} \quad M \mapsto PM^T Q,$$
where $P$ and $Q$ are non-singular matrices and $M^T$ denotes the transpose of $M$, and that Dieudonn\'e's theorem states that
the Frobenius automorphisms are the only linear maps from $\Mat_n(\K)$ to itself which map the set of invertible matrices onto itself:

\begin{theo}\label{strong}
Let $\calV$ be a linear subspace of $\Mat_n(\K)$ such that $\codim \calV<n-1$.
Let $f : \calV \hookrightarrow \Mat_n(\K)$ be a linear embedding such that
$$\forall M \in \calV, \; f(M) \in \GL_n(\K) \Leftrightarrow M \in \GL_n(\K).$$
Then, $f$ extends to a Frobenius automorphism of $\Mat_n(\K)$ unless
$n=3$, $\codim \calV=1$ and $\K \simeq \F_2$.
\end{theo}

\begin{theo}
Let $\calV$ be a linear subspace of $\Mat_n(\K)$ such that $\codim \calV<n-1$.
Let $f : \calV \rightarrow \calV$ be a linear bijection such that $f\bigl(\calV \cap \GL_n(\K)\bigr) \subset \GL_n(\K)$.
Then, $f$ extends to a Frobenius automorphism of $\Mat_n(\K)$ unless
$n=3$, $\codim \calV=1$ and $\K \simeq \F_2$.
\end{theo}

\begin{theo}
Assume that $\K$ is infinite.
Let $\calV$ be a linear subspace of $\Mat_n(\K)$ such that $\codim \calV<n-1$,
and $f : V \hookrightarrow \Mat_n(\K)$ be a linear embedding such that $f\bigl(\calV \cap \GL_n(\K)\bigr) \subset \GL_n(\K)$.
Then, $f$ extends to a Frobenius automorphism of $\Mat_n(\K)$.
\end{theo}

\subsection{Further notation}

For $M \in \Mat_{n,p}(\K)$,
we denote by $m_{i,j}$ its entry on the $i$-th row and $j$-th column.
The image - or column space - of $M$ is denoted by $\im M$.
The kernel - or null space - of $M$ is denoted by $\Ker M$.
If $M$ is a square matrix, we denote by $\tr(M)$ its trace and by $\widetilde{M}$ the transpose of the matrix of cofactors of $M$, and we recall
the formula $\widetilde{M}=\det(M)\,M^{-1}$ when $M$ is invertible.

We denote by $\frak{sl}_n(\K)$ the subspace of matrices with trace $0$ in $\Mat_n(\K)$, and by
$T_n^+(\K)$ (respectively, by $T_n^-(\K)$) the subspace of upper-triangular matrices (respectively, of lower-triangular matrices).
We equip $\Mat_{n,p}(\K)$ with the non-degenerate symmetric bilinear form
$b : (A,B) \mapsto \tr(A^TB)$. Given a subset $\calA$ of $\Mat_{n,p}(\K)$, its orthogonal subspace with respect to $b$
will always be denoted by $\calA^\bot$ unless specified otherwise.

\noindent We make the group $\GL_n(\K) \times \GL_p(\K)$ act on the set of linear subspaces of $\Mat_{n,p}(\K)$
by
$$(P,Q).\calV:=P\,\calV\,Q^{-1}.$$
Two linear subspaces of the same orbit will be called \textbf{equivalent}
(this means that they represent, in a change of bases, the same set of linear transformations from a $p$-dimensional vector space
to an $n$-dimensional vector space). They will be called \textbf{similar} if $n=p$ and we can take $P=Q$ in the above condition.

Given a non-empty subset $\calV$ of $\Mat_{n,p}(\K)$, we denote by $\rk \calV$ the maximal rank for a matrix in $V$, and call it
the \textbf{rank} of $\calV$.

Given subsets of $\Mat_n(\K)$ and $\Mat_p(\K)$, respectively, we set
$$\calA \vee \calB:=\Biggl\{ \begin{bmatrix}
A & C \\
0 & B
\end{bmatrix} \mid (A,B,C) \in \calA \times \calB \times \Mat_{n,p}(\K)\Biggr\},$$
which is a subset of $\Mat_{n+p}(\K)$.

For $(s,t)\in \lcro 0,n\rcro \times \lcro 0,p\rcro$, we define
$$\calR(s,t):=\biggl\{\begin{bmatrix}
M & N \\
P & [0]_{(n-s)\times (p-t)}
\end{bmatrix} \mid M \in \Mat_{s,t}(\K), \; N \in \Mat_{s,p-t}(\K), P \in \Mat_{n-s,t}(\K)\biggr\},$$
which is a linear subspace of $\Mat_{n,p}(\K)$.
Notice that we understate $n$ and $p$ in this notation; however, no confusion should arise when we use it.

In particular, $\calR(r,0)$ is the set of all matrices with all rows zero starting from the $(r+1)$-th,
and $\calR(0,r)$ is the set of all matrices with all columns zero starting from the $(r+1)$-th,
and hence they are $\overline{r}$-subspaces of $\Mat_{n,p}(\K)$.
More generally, $\calR(s,t)$ is always an $\overline{s+t}$-space.

\subsection{Main results}

Here is our first theorem for $\overline{r}$-subspaces of square matrices, generalizing Atkinson and Lloyd's theorem to all fields:

\begin{theo}\label{square}
Let $\K$ be an arbitrary field, and $n$ be a positive integer. Let $r \in \lcro 1,n-1\rcro$ and
$\calV$ be an $\overline{r}$-subspace of $\Mat_n(\K)$.
\begin{enumerate}[(a)]
\item If $\dim \calV>nr-r+1$, then $\calV$ is equivalent to a linear subspace of $\calR(r,0)$ or $\calR(0,r)$,
i.e.\ either there exists an $r$-dimensional subspace $F$ of $\K^n$ such that $\forall M \in \calV, \; \im M \subset F$,
or there exists an $(n-r)$-dimensional subspace $G$ of $\K^n$ such that $\forall M \in \calV, \; G \subset \Ker M$.
\item If $\dim \calV=nr-r+1$ and $(n,r,\card \K) \neq (3,2,2)$, then either $\calV$ is equivalent to a linear subspace of $\calR(r,0)$ or $\calR(0,r)$,
or $\calV$ is equivalent to $\calR(1,r-1)$ or $\calR(r-1,1)$.
\end{enumerate}
\end{theo}

Notice that $\dim \calR(1,r-1)=\dim \calR(r-1,1)=n\,r-r+1$, whence the additional new cases when $\dim \calV=nr-r+1$.

In the special case when $n=3$, $r=2$ and $\K \simeq \F_2$, the following counter-example was brought forth in an article
of Meshulam \cite{Meshulam}:
the linear subspace
$$\calJ_3(\F_2):=T_3^-(\F_2) \cap \frak{sl}_3(\F_2)=\Biggl\{
\begin{bmatrix}
a & 0 & 0 \\
c & b & 0 \\
d & e & a+b
\end{bmatrix} \mid (a,b,c,d,e)\in \F_2^5\Biggr\}$$
of $\Mat_3(\F_2)$ has rank $2$, dimension $5$ but it is an easy exercise
to prove that it is neither equivalent to $\calR(1,1)$
nor to a linear subspace of $\calR(2,0)$ or $\calR(0,2)$:
notice that, given some $x \in \F_2^3 \setminus \{0\}$, the linear subspace $\calJ_3(\F_2)x$ can have any dimension between $1$ and $3$
(but never $0$), depending on $x$. As we shall see, this counter-example is exceptional. We will prove the following theorem indeed:

\begin{theo}\label{M3F2}
Let $\calV$ be a $\overline{2}$-subspace of $\Mat_3(\F_2)$ with dimension $5$. Then:
\begin{enumerate}[(i)]
\item Either $\calV$ is equivalent to a linear subspace of $\calR(0,2)$ or $\calR(2,0)$;
\item Or $\calV$ is equivalent to $\calR(1,1)$;
\item Or $\calV$ is equivalent to $\calJ_3(\F_2)$.
\end{enumerate}
\end{theo}

\vskip 3mm
\noindent With the same techniques, we
shall also establish the following theorem for spaces of rectangular matrices,
already proved by Beasley \cite{Beasley} in the case $\card \K>r$.

\begin{theo}\label{rectangular}
Let $\K$ be an arbitrary field, and $n$ and $p$ be positive integers with $n>p$. Let $r \in \lcro 1,p-1\rcro$
and $\calV$ be an $\overline{r}$-subspace of $\Mat_{n,p}(\K)$.
\begin{enumerate}[(a)]
\item If $\dim \calV>nr-r+1+p-n$, then $\calV$ is equivalent to a linear subspace of $\calR(0,r)$.
\item If $\dim \calV=nr-r+1+p-n$, then either $\calV$ is equivalent to a linear subspace of $\calR(0,r)$,
or it is equivalent to $\calR(1,r-1)$, or it is equivalent to $\calR(r,0)$ and then $n=p+1$ or $r=1$.
\end{enumerate}
\end{theo}

\noindent Applying the transposition shows that the previous theorems encompass $\Mat_{n,p}(\K)$ for every pair $(n,p)$ of positive integers.

\vskip 2mm
In order to prove the above theorems, it is necessary to avoid using polynomials of large degrees, unlike Flanders, Atkinson, Lloyd and Beasley.
This basically forces us to use only very elementary tools of linear algebra, such as the rank theorem,
gaussian elimination, the factorization lemma for linear mappings, and elementary block matrix computations.
The term ``matrix combinatorics" is probably a good way to describe that sort of technique.

\paragraph{}
Among the various lemmas involved in solving the above problems, some are genuinely interesting for their own sake,
so we shall highlight them here:

\begin{lemme}[Inverse transitivity lemma]\label{geninverse}
Let $\calV$ be an affine subspace of $\Mat_n(\K)$ such that $\codim \calV<n-1$.
Then, for every $x \in \K^n \setminus \{0\}$,
$$\Vect\{A^{-1}x \mid A \in \calV \cap \GL_n(\K)\}=\K^n$$
and hence
$$\Vect\{\widetilde{A}x \mid A \in \calV\}=\K^n.$$
\end{lemme}
\noindent Note how this strengthens the part of Dieudonn\'e's theorem that states that $\calV \cap \GL_n(\K)$ is non-empty.

\begin{lemme}[Representation lemma]\label{reprtheorem}
Let $n$, $p$ and $r$ be positive integers.  Let $\calV$ be a linear subspace of
$\Mat_{n,r}(\K)$ such that $\dim \calV \geq n\,r-n+2$. Let
$\varphi : \calV \rightarrow \Mat_{n,p}(\K)$ be a linear map such that
$\im \varphi(M) \subset \im M$ for every $M \in \calV$. \\
Then, there exists $C \in \Mat_{r,p}(\K)$ such that $\varphi(M)=MC$ for every $M \in \calV$.
\end{lemme}

Besides being a major key both in our proofs of the above theorems and in the proof of Theorem \ref{strong}
from \cite{largedimpres}, Lemma \ref{reprtheorem} is particularly interesting for its connection with the currently
fashionable topic of algebraic reflexivity. Remember that, given vector spaces $U$ and $V$, a vector space $\calS$ of linear operators from $U$ to $V$
is called \textbf{algebraically reflexive} when, for every linear map $f : U \rightarrow V$, the condition
$\forall x \in U, \; f(x) \in \calS x$ implies $f \in \calS$.
The above representation lemma can be turned into a theorem giving a sufficient condition for algebraic reflexivity:

\begin{theo}\label{algrefl}
Let $U$ and $V$ be finite-dimensional vector spaces, and $\calS$ be a linear subspace of $\calL(U,V)$.
Set $U_0:=\underset{f \in \calS}{\bigcap} \Ker f$. If $\dim U-\dim U_0 \geq \dim \calS \dim V-\dim V+2$,
then $\calS$ is algebraically reflexive.
\end{theo}

After writing this article, we found out that the lower bound $\dim \calS \dim V-\dim V+2$ in Theorem \ref{algrefl}
is not optimal. On one hand, every $1$-dimensional space of operators is algebraically reflexive.
On the other hand, when $\dim \calS>1$, the optimal lower bound on $\dim U-\dim U_0$ happens to be $\dim \calS \dim V-2\dim V+3$
if $\K$ has more than three elements, and $\dim \calS \dim V-2\dim V+4$ otherwise.
We shall not include a proof of those statements because they are both very different and far longer than
the one of Theorem \ref{algrefl}.

\vskip 4mm
In the last section of the paper, we shall prove the following theorem, which roughly doubles the
range of high dimensions for which the structure of $\overline{r}$-spaces is known,
with a restriction on the cardinality of the underlying field, however. 

\begin{theo}[Second classification theorem]\label{class2theorem}
Let $n$, $p$ and $r$ be positive integers with $n\geq p >r$.
Let $\calV$ be an $\overline{r}$-subspace of $\Mat_{n,p}(\K)$ such that $\dim \calV \geq nr-2r+4+2(p-n)$.
Assume that $\# \K>r$.
Then, $\calV$ is equivalent to a subspace of one of the spaces
$\calR(0,r)$, $\calR(r,0)$, $\calR(1,r-1)$, $\calR(r-1,1)$, $\calR(2,r-2)$ or $\calR(r-2,2)$. \\
\end{theo}

Note that $\calR(2,r) \subset \Mat_{n,p}(\K)$ has dimension $nr-2r+4+2(p-n)$.
For small finite fields, we suspect that many exceptional cases should arise between dimensions $nr-2r+4+2(p-n)$ and $nr-r+1+(p-n)$.

We finish by stating what should be the ultimate conjecture on the structure of large $\overline{r}$-spaces of square matrices:

\begin{conj}\label{ultimateconjecture}
Let $n$ and $r$ be positive integers with $n>r$.
Let $\calV$ be an $\overline{r}$-subspace of $\Mat_n(\K)$ such that $\dim \calV \geq nr-\lfloor \frac{r}{2} \rfloor\bigl(r-
\lfloor \frac{r}{2} \rfloor\bigr)$.
Then, $\calV$ is equivalent to a subspace of $\calR(s,r-s)$ for some $s \in \lcro 0,r\rcro$.
\end{conj}

The lower bound $nr-\lfloor \frac{r}{2} \rfloor\bigl(r-\lfloor \frac{r}{2} \rfloor\bigr)$
is the minimal dimension for a subspace of $\Mat_n(\K)$ of type $\calR(s,r-s)$ with $s \in \lcro 0,r\rcro$.

\subsection{Structure of the article}
Our proof of Theorems \ref{square} and \ref{rectangular} has two major steps. We will show that, in most cases, we can
use a transposition and right and left-multiplication by non-singular matrices to reduce $\calV \subset \Mat_{n,p}(\K)$ to the form
$\biggl\{\begin{bmatrix}
M & \varphi(M)
\end{bmatrix} \mid M \in \calW\biggr\}$, where $\calW$ is a linear subspace of $\Mat_{n,r}(\K)$ and
$\varphi$ is a linear map. The second step is to prove, with the assumption $\dim \calW \geq nr-n+2$,
that the matrices of $\calV$ vanish on some common linear subspace of dimension $p-r$:
we will coin this as the Common Kernel Theorem.
The first step will use some recent new ideas for proving Flanders's theorem (see \cite{affpres}).
The Common Kernel Theorem is completely independent from the first step and uses the affine version of Flanders's theorem (again, see
\cite{affpres}): since its proof involves no discussion of special cases, we shall start with it
(see Section \ref{commonkersection}, which features the proof of the representation lemma and the derivation of Theorem \ref{algrefl}),
then work on the reduction to the situation of the Common Kernel Theorem, both for square matrices
and rectangular matrices (Sections \ref{reductionI} and \ref{reductionII}). The case $\dim \calV=n\,r-r+1+p-n$ is a lot more involving
than the case of sharp inequality (although it is based on the same core ideas) so we will devote
the entire Section \ref{reductionII} to its study.
In Section \ref{M3F2section}, we will classify the $5$-dimensional $\overline{2}$-subspaces of $\Mat_3(\F_2)$
(there, we will use various results from the previous sections).

The final section is devoted to the proof of the second classification theorem (Theorem \ref{class2theorem}).
This part is largely independent from the rest, save for the use of Lemma \ref{geninverse} and of basic matrix identities that are obtained in Section \ref{start}.

\section{The Common Kernel Theorem}\label{commonkersection}

\subsection{Statement of the theorem, and the structure of its proof}

This section is devoted to the proof of the following theorem, which is a major tool for establishing Theorems \ref{square} and \ref{rectangular}
but is also quite interesting in itself.

\begin{theo}[Common Kernel Theorem]\label{comkertheo}
Let $n$, $p$ and $r$ be three positive integers with $n>r$ and $p>r$. Let $\calW$ be a linear subspace of
$\Mat_{n,r}(\K)$ such that $\dim \calW \geq n\,r-n+2$. Let
$\varphi : \calW \rightarrow \Mat_{n,p-r}(\K)$ be a linear map.
Assume that
$$\calV:=\Bigl\{\begin{bmatrix}
M & \varphi(M)
\end{bmatrix}\mid M \in \calW\Bigr\} \subset \Mat_{n,p}(\K)$$
is an $\overline{r}$-space, and that $(n,r,\card \K) \neq (3,2,2)$ or $\dim \calW>n\,r-n+2$.
Then, there exists a $(p-r)$-dimensional linear subspace $G$ of $\K^p$ such that $G \subset \Ker N$ for every $N \in \calV$, i.e.\
$\calV$ is equivalent to a linear subspace of $\calR(0,r)$.
\end{theo}

\begin{Rems}
\begin{enumerate}[(a)]
\item Note that $\calJ_3(\F_2)$ is a counter-example in the case when $(n,r,\card \K)=(3,2,2)$ and $\dim \calW=5$.
\item Notice also that the lower bound $n\,r-n+2$ on $\dim \calW$ is tight if $r \geq 2$.
Consider indeed the linear subspace $\calW:=\calR(1,r-1) \subset \Mat_{n,r}(\K)$, which has dimension $n\,r-n+1$, and define, for $M \in \calW$,
$$\varphi(M):=\begin{bmatrix}
m_{2,1} & [0]_{1 \times (p-r-1)} \\
[0]_{(n-1) \times 1} & [0]_{(n-1) \times (p-r-1)}
\end{bmatrix} \in \Mat_{n,p-r}(\K).$$
A straightforward computation shows that $\rk \calV \leq r$ and $\underset{M \in \calV}{\bigcap}\,\Ker M$ has dimension $p-r-1$.
Therefore, $\calV$ satisfies all the assumptions of the Common Kernel Theorem but it is not equivalent to a subspace of $\calR(0,r)$.
\end{enumerate}
\end{Rems}

\noindent An important step of the proof will involve a recent result from \cite{affpres}, which
is an extension of Flanders's theorem to affine subspaces:

\begin{theo}\label{FlandersdSP}
Given positive integers $n \geq p$, let $\calV$ be an affine subspace of $\Mat_{n,p}(\K)$
such that $\rk \calV<p$.
Then, $\dim \calV \leq n(p-1)$. \\
If in addition $\dim \calV=n(p-1)$ and $(n,p,\card \K) \neq (2,2,2)$, then
$\calV$ is a linear subspace of $\Mat_{n,p}(\K)$.
\end{theo}

\noindent We will use the following corollary of Theorem \ref{FlandersdSP}:

\begin{cor}\label{corgen}
Let $n$ and $p$ be positive integers such that $n \geq p$.
Let $\calV$ be a linear subspace of $\Mat_{n,p}(\K)$ such that $\dim \calV>np-n$,
and assume that $(n,p,\card \K) \neq (2,2,2)$ or $\dim \calV>np-n+1$. Then, $\calV$ is spanned by its rank $p$ matrices.
\end{cor}

\begin{proof}
Assuming that the contrary holds, there would be a linear hyperplane $H$ of $\calV$ containing all the rank $p$ matrices of $\calV$.
Choosing $M_0 \in \calV \setminus H$, it would follow that $\rk(M_0+H)<p$, whereas $M_0+H$ is an affine subspace
of $\Mat_{n,p}(\K)$ with dimension greater than or equal to $n(p-1)$. As $M_0+H$ is not a linear subspace of $\Mat_{n,p}(\K)$,
this would contradict Theorem \ref{FlandersdSP} unless $(n,p,\# \K)=(2,2,2)$, in which case the contradiction would come from the
fact that $\dim(M_0+H)=\dim \calV-1>n(p-1)$.
\end{proof}

\begin{Rem}
Notice the exceptional case of
$$T_2^+(\F_2)=\biggl\{\begin{bmatrix}
a & b \\
0 & c
\end{bmatrix} \mid (a,b,c)\in \F_2^3\biggr\}$$
which has dimension $3$ but is not spanned by its non-singular elements (there are only two of them!).
\end{Rem}

\vskip 3mm
\noindent Theorem \ref{comkertheo} will obviously ensue should we prove Propositions \ref{comstep1}
and \ref{comstep2} below:

\begin{prop}\label{comstep1}
With the assumptions from Theorem \ref{comkertheo}, one has
$$\forall M \in \calW, \; \im \varphi(M) \subset \im M.$$
\end{prop}

\begin{prop}\label{comstep2}
Let $n$, $p$ and $r$ be three positive integers such that $p>r$. Let $\calW$ be a linear subspace of
$\Mat_{n,r}(\K)$ such that $\dim \calW \geq n\,r-n+2$. Let
$\varphi : \calW \rightarrow \Mat_{n,p-r}(\K)$ be a linear map.
Consider the linear subspace
$$\calV=\Bigl\{\begin{bmatrix}
M & \varphi(M)
\end{bmatrix}\mid M \in \calW\Bigr\},$$
and assume that $\im \varphi(M) \subset \im M$ for all $M \in \calW$.
Then, there exists a $(p-r)$-dimensional subspace $G$ of $\K^p$ such that $G \subset \Ker N$ for every $N \in \calV$.
\end{prop}

Proposition \ref{comstep2} is a corollary of Lemma \ref{reprtheorem} from the introduction.
In the rest of the section, we shall successively prove Proposition \ref{comstep1},
then Lemma \ref{reprtheorem}, and then derive Proposition \ref{comstep2} from it. We will conclude by explaining
how Theorem \ref{algrefl} on algebraic reflexivity follows from Lemma \ref{reprtheorem}.

\subsection{Proof of Proposition \ref{comstep1}}
Under the assumptions of Theorem \ref{comkertheo},
consider the linear subspace $\calW'$ consisting of the matrices in $\calW$ having the form
$M=\begin{bmatrix}
N \\
[0]_{1 \times r}
\end{bmatrix}$ for some $N \in \Mat_{n-1,r}(\K)$.
Then, we find a linear subspace $\calH$ of $\Mat_{n-1,r}(\K)$ and a linear isomorphism
$i : \calH \overset{\simeq}{\rightarrow} \calW'$ such that
$i(N)=\begin{bmatrix}
N \\
[0]_{1 \times r}
\end{bmatrix}$ for every $N \in \calH$.
We have linear maps $C$ and $\alpha$ defined on $\calH$ such that
$$\forall N \in \calH, \; \varphi(i(N))=\begin{bmatrix}
C(N) \\
\alpha(N)
\end{bmatrix} \quad \text{with $C(N) \in \Mat_{n-1,p-r}(\K)$ and $\alpha(N) \in \Mat_{1,p-r}(\K)$.}$$
Given $N \in \calH$ such that $\rk N=r$,
the fact that $\rk \begin{bmatrix}
i(N) & \varphi(i(N))
\end{bmatrix} \leq r$ reads
$$\rk \begin{bmatrix}
N & C(N) \\
[0]_{1 \times r} & \alpha(N)
\end{bmatrix} \leq r,$$
and hence $\alpha(N)=0$.
However, the rank theorem shows that $\dim \calH=\dim \calW' \geq \dim \calW-r>(n-1)(r-1)$.
It follows from Corollary \ref{corgen} that $\calH$ is spanned its rank $r$ matrices,
which yields $\alpha=0$ (notice, in the exceptional case when $n=3$, $r=2$ and $\K \simeq \F_2$, that the assumptions of Theorem \ref{comkertheo}
ensure that $\dim \calH \geq (n-1)(r-1)+2$).
To sum up, we have proved that for every $M \in \calW$, the condition $\im M \subset \K^{n-1} \times \{0\}$ implies
$\im \varphi(M) \subset \K^{n-1} \times \{0\}$.

Notice that the assumptions remain essentially unchanged should $\calW$ be replaced with
$P\,\calW$ for an arbitrary non-singular matrix $P \in \GL_n(\K)$ (and $\varphi$ replaced with $M \mapsto P\varphi(P^{-1}M)$).
As the natural action of $\GL_n(\K)$ on the
set of linear hyperplanes of $\K^n$ is transitive, we deduce that for every linear hyperplane $H$ of $\K^n$
and for every $M \in \calW$, the condition
$\im M \subset H$ implies $\im \varphi(M) \subset H$.

Finally, let $M \in \calW$. Writing $\im M$ as the intersection of a family of linear hyperplanes of $\K^n$, we
deduce from the above result that $\im \varphi(M) \subset \im M$. Thus, Proposition \ref{comstep1} is proved.

\subsection{Proof of Lemma \ref{reprtheorem}}

We perform an induction on $r$, with $n$ and $p$ fixed.
The case $r=0$ is trivial (the case $r=1$ is also easy but the reader will carefully check that we
actually start from $r=0$).
Given some positive integer $r$, assume that Lemma \ref{reprtheorem} holds for $(n,r-1,p)$.
Let $\calV$ and $\varphi$ be as in Lemma \ref{reprtheorem} for the triple $(n,r,p)$.
In $\calV$, consider the linear subspace $\calW$ of matrices of the form
$M=\begin{bmatrix}
[0]_{n \times 1} & K(M)
\end{bmatrix}$ with $K(M) \in \Mat_{n,r-1}(\K)$ (i.e.\ the matrices of $\calV$ with first column zero).
The rank theorem shows that $\dim \calW \geq \dim \calV-n \geq n(r-1)-n+2$.
Thus, $K(\calW)$ and the map $M \mapsto \varphi(K^{-1}(M))$ satisfy the assumptions of Lemma \ref{reprtheorem},
which yields a matrix $C \in \Mat_{r-1,p}(\K)$ such that
$\varphi(M)=K(M)C$ for every $M \in \calW$. \\
Setting $\widetilde{C}:=\begin{bmatrix}
[0]_{1 \times p} \\
C
\end{bmatrix} \in \Mat_{r,p}(\K)$ and replacing $\varphi$ with $M \mapsto \varphi(M)-M\widetilde{C}$,
we leave both the assumptions and the desired conclusion unchanged, so that no generality is lost in assuming that
$$\forall M \in \calW, \; \varphi(M)=0.$$
Now, set $F:=\bigl\{C_1(M) \mid M \in \calV\bigr\} \subset \K^n$, where $C_1(M)$ denotes the first column of $M$.
The rank theorem shows that
$$\dim F^\bot+\dim K(\calW)^\bot=\dim \Mat_{n,r}(\K)-\dim \calV \leq n-2.$$
We choose a basis $(x_1,\dots,x_s)$ of $F^\bot$ and a basis $(y_1,\dots,y_t)$ of the linear subspace of $K(\calW)^\bot$
spanned by its rank $1$ matrices. Then, $\im A \subset \im(y_1)+\cdots+\im(y_t)$ for every rank $1$ matrix $A$ in $K(\calW)^\bot$,
and $x \in \Vect(x_1,\dots,x_s)$ for every $x \in F^\bot$. Define finally
$$G:=\Vect(x_1,\dots,x_s)+\im(y_1)+\cdots+\im(y_t) \subset \K^n$$
and notice that $\dim G \leq s+t \leq \dim F^\bot+\dim K(\calW)^\bot \leq n-2$
(and that $G$ does not depend on the choice of $(y_1,\dots,y_t)$).
Now, we use the following lemma, the proof of which we postpone:

\begin{lemme}\label{nicebasis}
Let $E$ be an $n$-dimensional vector space, and
$H$ be a linear subspace of $E$ such that $\codim_E H \geq 2$. Then, there is a basis
$(e_1,\dots,e_n)$ of $E$ such that each plane $\Vect(e_1,e_2)$, $\Vect(e_2,e_3)$, \dots, $\Vect(e_{n-1},e_n)$ intersects $H$ trivially.
\end{lemme}

Replacing $\calV$ and $\varphi$ respectively with $P\,\calV$ and $M \mapsto P\,\varphi(P^{-1}M)$
for some well-chosen non-singular matrix $P \in \GL_n(\K)$,
and denoting by $(e_1,\dots,e_n)$ the canonical basis of $\K^n$, we may now assume that
each plane $\Vect(e_1,e_2)$, $\Vect(e_2,e_3)$, \dots, $\Vect(e_{n-1},e_n)$ intersects $G$ trivially.
In this situation, we can compute $\varphi$.
Let $A \in \calV$ be such that $a_{1,1}=0$, and write
$$A=\begin{bmatrix}
0 & L_1 \\
[?]_{(n-1) \times 1} & [?]_{(n-1) \times (r-1)}
\end{bmatrix} \quad \text{with $L_1 \in \Mat_{1,r-1}(\K)$.}$$
Since $e_1$ does not span the image of any $A \in K(\calW)^\bot$, some matrix of $K(\calW)$
has first row $L_1$, i.e.\ some matrix of $\calV$ has the form
$$A'=\begin{bmatrix}
0 & L_1 \\
[0]_{(n-1) \times 1} & [?]_{(n-1) \times (r-1)}
\end{bmatrix}.$$
Then, $\varphi(A')=0$ and hence
$$\varphi(A)=\varphi(A-A')=\varphi\begin{bmatrix}
0 & [0]_{1 \times (r-1)} \\
[?]_{(n-1) \times 1} & [?]_{(n-1) \times (r-1)}
\end{bmatrix}.$$
The assumptions of Lemma \ref{reprtheorem} show that the first row of $\varphi(A)$ must be $0$.
More generally, since none of the $e_i$'s belongs to $G$, we obtain that, for all $M \in \calV$ and all $i \in \lcro 1,n\rcro$,
the $i$-th row of $\varphi(M)$
is zero whenever $m_{i,1}=0$ (\emph{notice that this only uses the fact that none of
$e_1,\dots,e_n$ belongs to $G$}).
By the factorization lemma for linear maps, this yields
row matrices $R_1,\dots,R_n$ in $\Mat_{1,p}(\K)$ such that
$$\forall M \in \calV, \;
\varphi(M)=\begin{bmatrix}
m_{1,1}\,R_1 \\
m_{2,1}\,R_2 \\
\vdots \\
m_{n,1}\,R_n
\end{bmatrix}.$$
Let us prove that $R_1=R_2=\cdots=R_n$.
By performing the row operation $L_1 \leftarrow L_1-L_2$, we transform the pair
$(\calV,\varphi)$ into a new pair $(\calV',\varphi')$ which essentially has the same properties, save for the
assumption on the relationship between $e_1,\dots,e_n$ and the subspace $G'$ (which we associate with $\calV'$ as we associated $G$ with $\calV$).
Since $e_1-e_2 \not\in G$, we have $e_1 \not\in G'$
and the above arguments show that $m_{1,1}\,R_1-m_{2,1}\,R_2=0$ for every $M \in \calV$ for which $m_{1,1}-m_{2,1}=0$.
Besides, there exists some $M \in \calV$ such that $m_{1,1}=m_{2,1}=1$.
Indeed, if not, there would be a non-zero vector $X \in F^\bot \cap \Vect(e_1,e_2)$,
yielding $G \cap \Vect(e_1,e_2) \neq \{0\}$. We deduce that $R_1=R_2$. \\
More generally, for every $i \in \lcro 1,n-1\rcro$, using the row operation $L_i \leftarrow L_i-L_{i+1}$
and the fact that $G \cap \Vect(e_i,e_{i+1})=\{0\}$ shows that $R_i=R_{i+1}$. \\
Therefore, $\forall M \in \calV, \; \varphi(M)=MC$ for $C:=\begin{bmatrix}
R_1 \\
[0]_{(r-1)\times p}
\end{bmatrix}$. \\
Thus, the proof of Lemma \ref{reprtheorem} will be complete when we prove Lemma \ref{nicebasis}.

\begin{proof}[Proof of Lemma \ref{nicebasis}]
It suffices to tackle the case when $E=\K^n$ and $\codim_E H=2$. Since $\GL_n(\K)$ acts transitively on the set of $(n-2)$-dimensional linear subspaces of $E$, we may also assume that $H$ is the subspace defined by the following system of (independent) linear equations:
$$\underset{k=1}{\overset{\lfloor n/2\rfloor}{\sum}}x_{2k}=0 \quad ; \quad
\underset{k=0}{\overset{\lfloor(n-1)/2\rfloor}{\sum}}x_{2k+1}=0$$
where, for $t \in \R$, we have denoted by $\lfloor t\rfloor$ the greatest integer $k$ such that $k \leq t$.
It is then easily checked that the canonical basis $(e_1,\dots,e_n)$ of $\K^n$
satisfies the conclusion of Lemma \ref{nicebasis} for this particular space $H$.
\end{proof}

This completes the proof of Lemma \ref{reprtheorem}.

\subsection{Proof of Proposition \ref{comstep2}}

Obviously, $\varphi : \calW \rightarrow \Mat_{n,p-r}(\K)$ satisfies the assumptions of Lemma \ref{reprtheorem},
which yields a matrix $C \in \Mat_{r,p-r}(\K)$ such that $\varphi(M)=MC$ for every $M \in \calW$. \\
Setting $A:=\begin{bmatrix}
C \\
-I_{p-r}
\end{bmatrix} \in \Mat_{p,p-r}(\K)$, we deduce that $\forall N \in \calV, \; NA=0$. \\
Therefore, every matrix of $\calV$ vanishes everywhere on $\im(A)$, which has dimension $p-r$, visibly.

Thus, Corollary \ref{comstep2} is established, which finishes the proof of the Common Kernel Theorem.

\subsection{From the representation lemma to a sufficient condition for algebraic reflexivity}

In this short paragraph, we derive Theorem \ref{algrefl} on algebraic reflexivity
from Lemma \ref{reprtheorem}.
Let $U$, $V$ and $\calS$ be as in Theorem \ref{algrefl}. Fix a basis $(f_1,\dots,f_r)$
of $\calS$, and let $f : U \rightarrow V$ be a linear map such that
$f(x) \in \calS x$ for all $x \in U$, that is $f(x) \in \Vect(f_1(x),\dots,f_r(x))$ for all $x \in U$.
Fix a basis $\bfB$ of $V$, and set $n:=\dim V$.
For $x \in U$, set
$$M(x)=\Mat_{\bfB}(f_1(x),\dots,f_r(x)) \quad \text{and} \quad N(x)=\Mat_{\bfB}(f_1(x),\dots,f_r(x),f(x)).$$
Set $\calV:=M(U)$. Since $\forall x \in U, \; M(x)=0 \Rightarrow f(x)=0$, the factorization lemma for linear maps
yields a linear map $\varphi : \calV \rightarrow \K^n$ such that
$$\forall x \in U, \; N(x)=\begin{bmatrix}
M(x) & \varphi(M(x))
\end{bmatrix}.$$
The assumptions on $f$ show that $\varphi(M(x)) \in \im M(x)$ for all $x \in U$, that is
$\im \varphi(A) \subset \im A$ for all $A \in \calV$.
On the other hand, the rank theorem shows that $\dim \calV=\dim U-\dim U_0$, and hence
$\dim \calV \geq nr-n+2$. Therefore, Lemma \ref{reprtheorem} applies to the pair $(\calV,\varphi)$
and gives rise to a column vector $C=\begin{bmatrix}
\lambda_1 \\
\vdots \\
\lambda_r
\end{bmatrix}$ such that $\varphi(A)=AC$ for all $A \in \calV$. This shows that $f=\underset{k=1}{\overset{r}{\sum}}\lambda_k\,f_k \in \calS$.

Thus, Theorem \ref{algrefl} is established.

\section{Reduction to the Common Kernel Theorem (I)}\label{reductionI}

In this section, we tackle the case of sharp inequality in Theorems \ref{square} and \ref{rectangular}.
As we have already explained, the basic strategy is to reduce the situation to the one of the Common Kernel Theorem.
This reduction involves ideas from a recent new proof of Flanders's theorem \cite{affpres}, together with Lemma \ref{geninverse}
which was stated in the introduction. We shall start by proving Lemma \ref{geninverse}.

\subsection{Inverse transitivity for a large subspace of matrices}

First of all, let us recall the statement of Lemma \ref{geninverse}:

\begin{center}
Let $\calV$ be an affine subspace of $\Mat_n(\K)$ such that $\codim \calV < n-1$. \\
Then, for every non-zero vector $x \in \K^n \setminus \{0\}$, one has
$$\Vect \bigl\{ A^{-1}x \mid A \in \calV \cap \GL_n(\K)\bigr\}=\K^n.$$
\end{center}
Remark that the upper bound $n-1$ is tight: indeed, the linear subspace
$\calV:=\K \vee \Mat_{n-1}(\K)$ has codimension $n-1$ in $\Mat_n(\K)$ and, for $e_1:=\begin{bmatrix}
1 & 0 & \cdots & 0
\end{bmatrix}^T \in \K^n$, one finds $A^{-1}e_1 \in \K e_1$ for all $A \in \calV \cap \GL_n(\K)$.

\vskip 3mm
Now, let us prove Lemma \ref{geninverse} as a corollary to Dieudonn\'e's theorem on affine subspaces of singular matrices.

We denote by $(e_1,\dots,e_n)$ the canonical basis of $\K^n$.
Let $x \in \K^n \setminus \{0\}$.
Our aim is to prove that $\Vect\bigl\{A^{-1} x\mid A \in \calV \cap \GL_n(\K)\bigr\}=\K^n$.

Here, we consider the symmetric bilinear form $b : (A,B) \mapsto \tr(AB)$ on $\Mat_n(\K)$,
we denote by $V$ the translation vector space of $\calV$ and by
$V^\bot$ the orthogonal subspace of $V$ for $b$.
In $V^\bot$, the rank $1$ matrices span a linear subspace of which we choose a basis $(B_1,\dots,B_p)$
consisting of rank $1$ matrices. Set $F:=\underset{k=1}{\overset{p}{\sum}} \im B_k$. Then, $\dim F \leq p\leq \codim \calV \leq n-2$
and $\im B \subset F$ for every $B \in V^\bot$ such that $\rk B=1$.
Note that there is a basis $(f_1,\dots,f_n)$ of $\K^n$ of which no vector belongs to
$F$ (use Lemma \ref{nicebasis}, for example). Replacing $\calV$ with $P\,\calV\,Q$ for a well-chosen pair $(P,Q) \in \GL_n(\K)^2$,
we reduce the situation to the one where:
\begin{enumerate}[(i)]
\item $x=e_1$;
\item for every $i \in \lcro 1,n\rcro$, the subspace $V^\bot$ contains no matrix $B$ such that $\im B=\Vect(e_i)$.
\end{enumerate}
By assumption (ii) for $i=1$, we find that, for any $C \in \Mat_{n,1}(\K)$, there is a matrix in
$\calV$ with first column $C$. In particular, $\calV$ contains a matrix of the form
$$A_0=\begin{bmatrix}
1 & L_0 \\
[0]_{(n-1) \times 1} & K_0
\end{bmatrix}, \quad \text{where $L_0 \in \Mat_{1,n-1}(\K)$ and $K_0 \in \Mat_{n-1}(\K)$.}$$
Now, denote by $G$ the linear subspace of $V$ consisting of its matrices with first column zero.
We write every $M \in G$ as
$$M=\begin{bmatrix}
0 & L(M) \\
[0]_{(n-1) \times 1} & K(M)
\end{bmatrix}, \quad \text{where $L(M) \in \Mat_{1,n-1}(\K)$ and $K(M) \in \Mat_{n-1}(\K)$.}$$
The rank theorem show that $\dim K(G)>(n-1)(n-2)$, and hence Dieudonn\'e's theorem for affine subspaces
(see \cite{Dieudonne} or \cite{affpres}) shows that the affine subspace $K_0+K(G)$ of $\Mat_{n-1}(\K)$ contains a non-singular matrix.
It follows that there is a non-singular matrix $P \in \GL_{n-1}(\K)$ and a row matrix $L_1 \in \Mat_{1,n-1}(\K)$ such that
$$A_1=\begin{bmatrix}
1 & L_1 \\
[0]_{(n-1) \times 1} & P
\end{bmatrix} \in \calV.$$
Therefore, $A_1^{-1}x=e_1$. Using the same method on every column, we find that, for every $i \in \lcro 1,n\rcro$,
the affine subspace $\calV$ contains a non-singular matrix $A_i$ with $i$-th column $\begin{bmatrix}
1 & 0 & \cdots & 0
\end{bmatrix}^T$, to the effect that $A_i^{-1}x=e_i$. Therefore
$\Vect\{A^{-1}x \mid A \in \calV \cap \GL_n(\K)\}=\K^n$, as claimed. This completes the proof of Lemma \ref{geninverse}.

\subsection{The general starting point}\label{start}

Let $n$, $p$ and $r$ be positive integers such that $n \geq p>r$, and
$\calV$ be an $\overline{r}$-subspace of $\Mat_{n,p}(\K)$
such that $\dim \calV \geq nr-r+1+p-n$. Then, $\dim \calV>n(r-1)$, and hence Flanders's theorem
(see \cite{Meshulam} or \cite{affpres} for the generalization to an arbitrary field) forbids $\rk \calV \leq r-1$, which shows that $\rk \calV=r$.
Replacing $\calV$ with an equivalent subspace if necessary,
we see that no generality is lost in assuming that $\calV$ contains the matrix
$$J_r:=\begin{bmatrix}
I_r & [0]_{r \times (p-r)} \\
0_{(n-r) \times r} & [0]_{(n-r) \times (p-r)}
\end{bmatrix}.$$
Let $M=\begin{bmatrix}
P & C \\
L & \alpha
\end{bmatrix} \in \calV$, with blocks $P$, $L$, $C$ and $\alpha$ of size
$r\times r$, $(n-r)\times r$, $r\times (p-r)$ and $(n-r) \times (p-r)$, respectively.
In the rest of the proof, the block decompositions will have the same configuration unless specified otherwise.

Our basic tool is the formula
\begin{equation}
\label{comatrixformula}
\boxed{L \widetilde{P} C=\det(P)\,\alpha.}
\end{equation}
To get this identity, one computes that, given a pair $(i,j)\in \lcro 1,n-r\rcro \times \lcro 1,p-r\rcro$,
the determinant of the submatrix of $M$ obtained in selecting row indexes in $\lcro 1,r\rcro \cup \{i+r\}$ and
column indexes in $\lcro 1,r\rcro \cup \{j+r\}$ is
$$\begin{vmatrix}
P & T_j \\
R_i & \alpha_{i,j}
\end{vmatrix},$$
where $R_i$ is the $i$-th row of $L$, and $T_j$ the $j$-th column of $C$.
Classically, this determinant equals $-R_i \widetilde{P} T_j+\det(P) \alpha_{i,j}$.
The formula ensues by noting that $R_i \widetilde{P} T_j$ is the entry of $L \widetilde{P} C$ at the $(i,j)$-spot.

Let us write every $M \in \calV$ as
$M=\begin{bmatrix}
K(M) & C(M) \\
L(M) & \alpha(M)
\end{bmatrix}$, and set $\calW:=\Ker K$, i.e.\ $\calW$ is the linear subspace of $\calV$ consisting of its matrices of the form
$\begin{bmatrix}
[0]_{r \times r} & ? \\
? & ?
\end{bmatrix}$.

Let $A=\begin{bmatrix}
P_1 & C_1 \\
L_1 & \alpha_1
\end{bmatrix} \in \calV$.
For every $M \in \calW$, the matrix $A+M$ belongs to $\calV$ and has $P_1$ as upper-left block, whence \eqref{comatrixformula} shows that
$$(L(M)+L_1)\,\widetilde{P_1}\,(C(M)+C_1)=\det(P_1)\bigl(\alpha(M)+\alpha_1\bigr).$$
Subtracting \eqref{comatrixformula} applied to $A$, we deduce that
\begin{equation}\label{general}
\forall M \in \calW, \; L(M)\,\widetilde{P_1}\,C(M)=\det(P_1)\,\alpha(M)-L(M) \,\widetilde{P_1}\, C_1-L_1\, \widetilde{P_1}\,C(M).
\end{equation}
Notice that the left-hand side of \eqref{general} is a quadratic function $\varphi$ of $M$ on $\calW$, whereas the right-hand side is a linear one.
By computing the polar function of $\varphi$ as defined by $b_\varphi(M,N):=\varphi(M+N)-\varphi(M)-\varphi(N)$, we deduce that
\begin{equation}\label{bilin}
\forall (M,N) \in \calW^2, \; L(M)\,\widetilde{P_1}\,C(N)+L(N)\,\widetilde{P_1}\,C(M)=0.
\end{equation}
Now, consider the linear subspace $\calH$ consisting of the matrices of $\calW$ which have the form
$\begin{bmatrix}
[0]_{r \times r} & [?]_{r \times (p-r)} \\
[0]_{(n-r)\times r} & [?]_{(n-r)\times (p-r)}
\end{bmatrix}$, i.e.\ matrices with all first $r$ columns zero.

With the special case $A=J_r$, we note that identity \eqref{general} yields that the linear map $M \mapsto C(M)$ is one-to-one on $\calH$,
and hence $\dim \calH=\dim C(\calH)$. With the rank theorem, we deduce that
$$\dim \calV=\dim K(\calV)+\dim L(\calW)+\dim C(\calH).$$
Returning to the general case, identity \eqref{bilin} shows that
\begin{equation}\label{ultimatenil}
\boxed{\forall M \in \calW, \; \forall N \in \calH, \; \forall P \in K(\calV), \;  L(M)\,\,\widetilde{P}\,C(N)=0.}
\end{equation}
Notice in particular that $I_r \in K(\calV)$, to the effect that:
$$\forall M \in \calW, \; \forall N \in \calH, \; L(M)\,C(N)=0.$$
Notice finally that if we have $\calH=\{0\}$ (which is not always the case),
then the factorization lemma for linear maps
shows that $\calV$ has the form given in the Common Kernel Theorem, and hence
$\calV$ is equivalent to a subspace of $\calR(0,r)$ unless $(n,p,r,\card \K)=(3,3,2,2)$.

\vskip 3mm
\noindent Finally, we set
$$G:=\sum_{N\in \calH} \im C(N) \quad \text{and}  \quad q:=\dim G.$$
In the rest of the section, we will focus on the case when $\dim \calV>n\,r-r+1+p-n$; we
wait until Section \ref{reductionII} to tackle the case of equality
(in that prospect, the following two paragraphs will serve as a necessary warm-up).

\subsection{The case $\dim \calV>nr-r+1$ for square matrices}

Here, we assume that $n=p$ and $\dim \calV>nr-r+1$.
On one hand $\dim C(\calH) \leq q\,(n-r)$ since $\im C(N) \subset G$ for every $N \in \calH$;
on the other hand, every matrix of $L(\calW)$ vanishes everywhere on $G$, and hence $\dim L(\calW) \leq (r-q)\,(n-r)$.
We deduce that
$$\dim L(\calW)+\dim C(\calH) \leq r\,(n-r),$$
which yields
$$\dim K(\calV)>r^2-r+1.$$
With the additional assumption $C(\calH) \neq \{0\}$, we choose $C \in C(\calH) \setminus \{0\}$ and use Lemma \ref{geninverse} to obtain
$\underset{A \in K(\calV)}{\sum} \im \widetilde{A}C=\K^r$; then, identity \eqref{ultimatenil} shows that
$\forall M \in \calW, \; L(M)=0$. Thus, replacing $\calV$ with $\calV^T$ helps us see that no generality is lost in assuming that
$C(\calH)=\{0\}$, and hence $\calH=\{0\}$. In that case, the Common Kernel Theorem readily yields the desired conclusion.

\subsection{The case $\dim \calV>nr-r+1+p-n$ for non-square matrices}\label{nonsquarestrict}

Here, we assume that $n>p$ and $\dim \calV>nr-r+1+p-n$.
Then we have $\dim \calV>(n-1)r \geq p\,r$, and hence $L(\calW) \neq \{0\}$.
If $\dim K(\calV)>r^2-r+1$, then the line of reasoning of the preceding paragraph shows that
$\calH=\{0\}$.

If we now assume that $\dim K(\calV)\leq r^2-r+1$, then, as the above line of reasoning shows that $\dim C(\calH) \leq q\,(p-r)$
and $\dim L(\calW) \leq (r-q)\,(n-r)$, we find
$$\dim \calV \leq \dim K(\calV)+\dim L(\calW)+\dim C(\calH) \leq n\,r-(n-p)\,q-r+1.$$
Since $n-p \geq 1$, combining this inequality with our assumptions on $\dim \calV$ yields that $q=0$,
and hence $\calH=\{0\}$: thus, we may yet again conclude using the Common Kernel Theorem.

\vskip 3mm
Thus, the statements on the cases of sharp inequality of dimensions in Theorems \ref{square} and \ref{rectangular}
are established. In the next section, we delve into the case of equality.

\section{Reduction to the common kernel theorem (II)}\label{reductionII}

In this section, we keep all the assumptions from Section \ref{start} but
also assume that $\dim \calV=nr-r+1+p-n$. Our goal is to prove the following facts:
\begin{itemize}
\item If $n>p$, then either $\calV$ is equivalent to a subspace satisfying the assumptions of the Common Kernel Theorem,
or it is equivalent to $\calR(1,r-1)$, or it is equivalent to $\calR(r,0)$, in which case $r=1$ or $p=n-1$;
\item If $n=p$, then either $\calV$ or $\calV^T$ is equivalent to a subspace satisfying the assumptions of the Common Kernel Theorem,
or $\calV$ is equivalent to $\calR(1,r-1)$ or $\calR(r-1,1)$, or $n=3$, $r=2$, $\K \simeq \F_2$ and $\calV$ is equivalent to 
$\calJ_3(\K):=T_3^-(\K) \cap \frak{sl}_3(\K)$.
\end{itemize}
Notice that we do not discard the case $(n,r,\card \K)=(3,2,2)$ yet. 
This will be useful in the prospect of Section \ref{M3F2section}. \\
We will need two lemmas. Afterwards, we will
finish the proof of Theorems \ref{square} and \ref{rectangular}.

\subsection{Additional lemmas}

\begin{lemme}\label{inversestab}
Let $\calY$ be a linear subspace of $\Mat_r(\K)$ such that
$\dim \calY=r^2-r+1$ and $I_r \in \calY$. Assume that there is a non-zero vector $x \in \K^r$ such that
$\Vect \{\widetilde{M} x \mid M \in \calY\} \neq \K^r$.
Then, $r>1$ and $\calY$ is similar either to $\Mat_1(\K) \vee \Mat_{r-1}(\K)$ or to $\Mat_{r-1}(\K) \vee \Mat_1(\K)$.
\end{lemme}

In some sense, this lemma can be seen as an exploration of the case of equality in Lemma \ref{geninverse}.

\begin{proof}
If $r=1$, then $\calY=\Mat_1(\K)$, which contradicts the assumption that $\Vect \{\widetilde{M} x \mid M \in \calY\} \neq \K^r$
for some non-zero vector $x$. Thus, $r>1$.

The assumptions yield two non-singular matrices $Q_1$ and $Q_2$ such that, for every $M \in \calY$,
the upper-left $(r-1) \times (r-1)$-submatrix of $Q_1MQ_2$ is singular. In other words, if we write
$M=Q_1^{-1}\begin{bmatrix}
R(M) & [?]_{(r-1) \times 1} \\
[?]_{1 \times (r-1)} & ?
\end{bmatrix}Q_2^{-1}$ with $R(M) \in \Mat_{r-1}(\K)$, then $R(\calY)$ is an $\overline{r-2}$-subspace of $\Mat_{r-1}(\K)$.
On the other hand, the rank theorem yields
$$\dim R(\calY) \geq \dim \calY-(2r-1)=(r-1)(r-2),$$
and hence Theorem \ref{FlandersdSP} shows that $R(\calY)$ is equivalent to either $\calR(0,r-2)$ or $\calR(r-2,0)$.
\begin{itemize}
\item In the first case, we have found two $1$-dimensional subspaces $D$ and $D'$ such that every $M \in \calY$
maps $D$ into $D'$. As $\calY$ contains $I_r$, one deduces that $D=D'$, which shows that $\calY$ is similar to
a linear subspace of $\Mat_1(\K) \vee \Mat_{r-1}(\K)$. As $\dim \calY=\dim \bigl(\Mat_1(\K) \vee \Mat_{r-1}(\K)\bigr)$, one deduces that $\calY$
is similar to $\Mat_1(\K) \vee \Mat_{r-1}(\K)$.
\item In the second case, the same line of reasoning yields a linear hyperplane $H$ of $\K^r$
which is stable under all the elements of $\calY$, and one deduces that $\calY$ is similar to $\Mat_{r-1}(\K) \vee \Mat_1(\K)$.
\end{itemize}
\end{proof}

\begin{lemme}\label{comatrix2lemma}
Let $n$ and $p$ be positive integers, and set $\calW:=\Mat_n(\K) \vee \Mat_p(\K)$.
Then,
$$\forall x \in \K^{n+p} \setminus (\K^n \times \{0\}), \; \Vect \bigl\{P^{-1}x \mid P \in \calW \cap \GL_{n+p}(\K)\bigr\}=\K^{n+p}$$
and
$$\forall x \in (\K^n \times \{0\}) \setminus \{0\}, \; \Vect \bigl\{P^{-1}x \mid P \in \calW \cap \GL_{n+p}(\K)\bigr\}=\K^n \times \{0\}.$$
In particular, $\K^n \times \{0\}$ is the sole non-trivial linear subspace of $\K^{n+p}$
which is stable under $\widetilde{M}$ for all $M \in \Mat_n(\K) \vee \Mat_p(\K)$.
\end{lemme}

\begin{proof}
One checks that the inverses of the invertible matrices of $\calW$ are the matrices of the form
$\begin{bmatrix}
A & [?]_{n \times p} \\
[0]_{p \times n} & B
\end{bmatrix}$, with $(A,B) \in \GL_n(\K) \times \GL_p(\K)$.
Since $\GL_n(\K)$ acts transitively on $\K^n \setminus \{0\}$, one deduces the second claimed result.
As $\GL_p(\K)$ acts transitively on $\K^p \setminus \{0\}$, we also obtain that
$\Vect\{P^{-1} e_{n+1}\mid P \in \calW \cap \GL_{n+p}(\K)\bigr\}=\K^{n+p}$, where
$e_{n+1}$ denotes the $(n+1)$-th vector of the canonical basis of $\K^{n+p}$.
Finally, given $x \in \K^{n+p} \setminus (\K^n \times \{0\})$,
we see from the above description of the inverses of the matrices of $\calW \cap \GL_{n+p}(\K)$
that there exists $Q \in \calW \cap \GL_{n+p}(\K)$ such that
$x=Q^{-1} e_{n+1}$; as $\calW \cap \GL_{n+p}(\K)$ is obviously a multiplicative group, this yields
\begin{align*}
\Vect\{  P^{-1} x\mid P \in \calW \cap \GL_{n+p}(\K)\bigr\}
& =\Vect\{  (QP)^{-1} e_{n+1}\mid P \in \calW \cap \GL_{n+p}(\K)\bigr\} \\
& =\Vect\{  R^{-1} e_{n+1}\mid R \in \calW \cap \GL_{n+p}(\K)\bigr\}=\K^{n+p}.
\end{align*}
\end{proof}

\begin{lemme}\label{lastlemma}
Let $\calY$ be an $\overline{r}$-subspace of $\Mat_{n,p}(\K)$ such that $\dim \calY=n\,r-r+1+p-n$, with $n \geq p>r > 1$.
Assume that $K(\calY)=\Mat_1(\K) \vee \Mat_{r-1}(\K)$ and, for every $C_1 \in \Mat_{n-r,r-1}(\K)$ and $L_1 \in \Mat_{1,p-r}(\K)$, that
the subspace $\calY$ contains the matrix
$$\begin{bmatrix}
0 & [0]_{1 \times (r-1)} & L_1 \\
[0]_{(r-1) \times 1} & [0]_{(r-1) \times (r-1)} & [0]_{(r-1) \times (p-r)} \\
[0]_{(n-r) \times 1} & C_1 & [0]_{(n-r) \times (p-r)}
\end{bmatrix}.$$
Then:
\begin{enumerate}[(i)]
\item Either $\calY$ is equivalent to $\calR(1,r-1)$;
\item Or $n=p=3$, $r=2$, $\K \simeq \F_2$ and
$\calY$ is equivalent to
$$\calJ_3(\K)=\Biggl\{
\begin{bmatrix}
a & 0 & 0 \\
c & b & 0 \\
d & e & a+b
\end{bmatrix} \mid (a,b,c,d,e)\in \K^5\Biggr\}.$$
\end{enumerate}
\end{lemme}

\begin{proof}
Set $\calK_r:=\Mat_1(\K) \vee \Mat_{r-1}(\K)$.
Since $\dim \calY=n\,r-r+1+p-n$ and $\dim K(\calY)=r^2-r+1$, the rank theorem and the hypotheses show that $\Ker K$ is precisely the space
of all matrices of the form
$$M_{L,C}:=\begin{bmatrix}
0 & [0]_{1 \times (r-1)} & L \\
[0]_{(r-1) \times 1} & [0]_{(r-1) \times (r-1)} & [0]_{(r-1) \times (p-r)} \\
[0]_{(n-r) \times 1} & C & [0]_{(n-r) \times (p-r)}
\end{bmatrix}$$
with $(L,C) \in \Mat_{1,p-r}(\K) \times \Mat_{n-r,r-1}(\K)$.
By the factorization lemma, this yields linear maps $\beta : \calK_r \rightarrow \Mat_{n-r,1}(\K)$,
$\gamma : \calK_r \rightarrow \Mat_{r-1,p-r}(\K)$ and $\delta : \calK_r \rightarrow \Mat_{n-r,p-r}(\K)$ such that every $M \in \calY$
splits up as
\begin{center}
\mbox{
\psfrag{M}{$M=\begin{bmatrix}
\hskip 5mm & \hskip 5mm & \hskip 5mm & \hskip 5mm & \hskip 5mm & \hskip 5mm & \hskip 5mm & \hskip 5mm \\
\hskip 5mm & \hskip 5mm & \hskip 5mm & \hskip 5mm & \hskip 5mm & \hskip 5mm & \hskip 5mm & \hskip 5mm \\
\hskip 5mm & \hskip 5mm & \hskip 5mm & \hskip 5mm & \hskip 5mm & \hskip 5mm & \hskip 5mm & \hskip 5mm \\
\hskip 5mm & \hskip 5mm & \hskip 5mm & \hskip 5mm & \hskip 5mm & \hskip 5mm & \hskip 5mm & \hskip 5mm \\
\hskip 5mm & \hskip 5mm & \hskip 5mm & \hskip 5mm & \hskip 5mm & \hskip 5mm & \hskip 5mm & \hskip 5mm \\
\hskip 5mm & \hskip 5mm & \hskip 5mm & \hskip 5mm & \hskip 5mm & \hskip 5mm & \hskip 5mm & \hskip 5mm \\
\hskip 5mm & \hskip 5mm & \hskip 5mm & \hskip 5mm & \hskip 5mm & \hskip 5mm & \hskip 5mm & \hskip 5mm
\end{bmatrix}.
$}
\psfrag{?}{$?$}
\psfrag{beta}{$\beta(K(M))$}
\psfrag{gamma}{$\gamma(K(M))$}
\psfrag{delta}{$\delta(K(M))$}
\psfrag{K}{$K(M)$}
\includegraphics{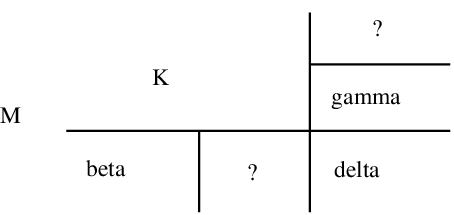}
}
\end{center}
\vskip 4mm
Now, let $P_0 \in \calK_r \cap \GL_r(\K)$, and choose $M \in \calY$ of the form
$M=\begin{bmatrix}
P_0 & C_0 \\
L_0 & \alpha_0
\end{bmatrix}$ (with the block format described in Section \ref{start}).
Then, formula \eqref{general} applied to $M$ and the matrix
$M_{L,0}$, for an arbitrary $L \in \Mat_{1,p-r}(\K)$, yields
$\delta(P_0)=L_0\,P_0^{-1}\,\begin{bmatrix}
L \\
[0]_{(r-1) \times (p-r)}
\end{bmatrix}$; taking $L=0$ shows that $\delta(P_0)=0$, and then taking a non-zero $L$ shows that
$\beta(P_0)=0$ since $P_0 \in \calK_r \cap \GL_r(\K)$. Symmetrically, applying formula \eqref{general} to $M$ and $M_{0,C}$, for an arbitrary $C \in
\Mat_{n-r,r-1}(\K)$, yields $\gamma(P_0)=0$.
We deduce that $\beta$, $\gamma$ and $\delta$ vanish everywhere on $\Vect(\calK_r \cap \GL_r(\K))$, and hence
on the whole $\calK_r$ if $(r,\card \K) \neq (2,2)$ (we may use again Corollary \ref{corgen}, although a more elementary proof can be given);
in the case $\delta$, $\beta$ and $\gamma$ are all zero, one deduces that $\calY$ is included in the space of all matrices of the form
$$\begin{bmatrix}
? & [?]_{1 \times (r-1)} & [?]_{1 \times (p-r)} \\
[0]_{(n-1) \times 1} & [?]_{(n-1) \times (r-1)} & [0]_{(n-1) \times (p-r)}
\end{bmatrix}.$$
As the dimensions of both vector spaces of matrices are equal, we deduce that this inclusion is an equality.
Permuting columns, one deduces that $\calY$ is equivalent to $\calR(1,r-1)$.
\vskip 3mm
From now on, we assume that $r=2$ and $\K=\F_2$.
Firstly, we prove that $\beta$ and $\gamma$ are zero.
Notice already that $\beta$, $\gamma$ and $\delta$ vanish on the matrices
$\begin{bmatrix}
1 & 0 \\
0 & 1
\end{bmatrix}$ and $\begin{bmatrix}
1 & 1 \\
0 & 1
\end{bmatrix}$, whose linear span is $\{M \in \calK_2 : \; \tr(M)=0\}$.
By the factorization lemma, we find three matrices
$C_0 \in \Mat_{n-2,1}(\F_2)$, $L_0 \in \Mat_{1,p-2}(\F_2)$ and $N_0 \in \Mat_{n-2,p-2}(\F_2)$ such that
$$\forall P \in \calK_2, \quad \beta(P)=(\tr P)\cdot C_0,\quad \gamma(P)=(\tr P)\cdot L_0 \quad \text{and} \quad
\delta(P)=(\tr P)\cdot N_0.$$
Applying this to the matrices $\begin{bmatrix}
1 & 0 \\
0 & 0
\end{bmatrix}$ and $\begin{bmatrix}
0 & 0 \\
0 & 1
\end{bmatrix}$
shows that, for all $(L,C)\in \Mat_{1,p-2}(\F_2) \times \Mat_{n-2,1}(\F_2)$, the space $\calY$ contains
$$\begin{bmatrix}
1 & 0 & L \\
0 & 0 & L_0 \\
C_0 & C & N_0
\end{bmatrix} \quad \text{and} \quad
\begin{bmatrix}
0 & 0 & L \\
0 & 1 & L_0 \\
C_0 & C & N_0
\end{bmatrix}.$$
If $L_0 \neq 0$, then taking $L=0$ and an arbitrary $C \neq 0$ in the first matrix
yields a matrix in $\calY$ with rank $\geq 3$, which is forbidden. Therefore, $L_0=0$. Using the same line of reasoning with the second
type of matrices, we find $C_0=0$. If $N_0=0$, then $\beta$, $\gamma$ and $\delta$ are all zero and we are done. \\
Now, we assume that $N_0 \neq 0$.
Taking $L=0$ and $C=0$ in the first type of matrices shows that $\rk N_0 \leq 1$, and hence $\rk N_0=1$.
If $n-2>1$, we may choose $L=0$ and $C \in \K^{n-2} \setminus \im N_0$ in the first type of matrices,
which yields a contradiction. Thus, $n=3$, and a similar line of reasoning shows that $p=3$.
It follows that $N_0=1$, and hence
$$\calY \subset \Biggl\{\begin{bmatrix}
a & c & d \\
0 & b & 0 \\
0 & e & a+b
\end{bmatrix} \mid (a,b,c,d,e)\in \F_2^5\Biggr\}.$$
The dimensions being equal on both sides, we deduce that the above inclusion is an equality.
Finally, by a series of obvious row and column operations, we see that $\calY$ is equivalent to the subspace
$$\Biggl\{\begin{bmatrix}
a & d & c \\
0 & a+b & e \\
0 & 0 & b
\end{bmatrix} \mid (a,b,c,d,e)\in \F_2^5\Biggr\}$$
and then to
$$\Biggl\{\begin{bmatrix}
b & 0 & 0 \\
e & a+b & 0 \\
c & d & a
\end{bmatrix} \mid (a,b,c,d,e)\in \F_2^5\Biggr\}=\calJ_3(\F_2),$$
which finishes the proof.
\end{proof}

\subsection{The case $\dim \calV=nr-r+1$ for square matrices}\label{limitsquare}

Here, we assume that $n=p$, and hence $\dim \calV=n\,r-r+1$. Once again, the assumption
$\dim K(\calV)>r^2-r+1$ would yield, using Lemma \ref{geninverse}, that $L(\calW)=\{0\}$ or $C(\calH)=\{0\}$, and in either case
we could use the Common Kernel Theorem to conclude immediately (save for the exceptional case when $(n,r,\card \K)=(3,2,2)$).
Thus, we shall assume that $L(\calW) \neq \{0\}$, $C(\calH) \neq \{0\}$ and $\dim K(\calV) \leq r^2-r+1$, in the rest of the proof.
Recall the notation $G=\sum_{N \in \calH} \im C(N)$ and $q=\dim G$.
Again, $\dim C(\calH) \leq q\,(n-r)$ and $\dim L(\calW) \leq (r-q)\,(n-r)$, and now the rank theorem shows that
\begin{align*}
nr-r+1  =\dim \calV & \leq \dim K(\calV)+\dim L(\calW)+\dim C(\calH) \\
& \leq r^2-r+1+(r-q)\,(n-r)+q\,(n-r)=nr-r+1.
\end{align*}
It follows that $\dim K(\calV)=r^2-r+1$, $\dim L(\calW)=(r-q)\,(n-r)$ and $\dim C(\calH)=q\,(n-r)$,
which in turns proves:
\begin{enumerate}[(i)]
\item That $C(\calH)$ is the set of all matrices of $\Mat_{r,n-r}(\K)$ whose image is included in $G$;
\item That $L(\calW)$ is the set of all matrices of $\Mat_{n-r,r}(\K)$ vanishing everywhere on $G$.
\end{enumerate}
Then, we deduce from identity \eqref{ultimatenil} that $K(\calV)$ satisfies the assumptions of Lemma \ref{comatrix2lemma},
and therefore $r>1$. Moreover, one sees from identity \eqref{ultimatenil} that $G$ is stable under $\widetilde{P}$ for all $P \in K(\calV)$.
Using Lemma \ref{inversestab}, one deduces that $K(\calV)$ is similar either to $\Mat_1(\K) \vee \Mat_{r-1}(\K)$ or to
$\Mat_{r-1}(\K) \vee \Mat_1(\K)$ and, in any case, Lemma \ref{comatrix2lemma} shows that $G$ is the sole non-trivial linear subspace of $\K^r$
which is stable under $\widetilde{P}$ for all $P \in K(\calV)$. Thus, we have $q=1$ in the first case, and $q=r-1$ in the second one.

Assume that $q=1$, so that $K(\calV)$ is similar to $\Mat_1(\K) \vee \Mat_{r-1}(\K)$.
Then, we see that no generality is lost in assuming that
$G=\K \times \{0\}$ and $K(\calV)=\Mat_1(\K) \vee \Mat_{r-1}(\K)$
(since none of the previous assumptions is modifying in replacing $\calV$ with $(P \oplus I_{n-r}) \calV (P \oplus I_{n-r})^{-1}$
for some $P \in \GL_r(\K)$).
Now, let us define $\calH'$ as the linear subspace of $\calW$ consisting of its matrices which have the form
$\begin{bmatrix}
[0]_{r\times r} & [0]_{r \times (n-r)} \\
[?]_{(n-r) \times r} & [?]_{(n-r) \times (n-r)}
\end{bmatrix}$. If $\calH'=\{0\}$, then the Common Kernel Theorem applied to
$\calV^T$ shows again that $\calV$ is equivalent to a subspace of $\calR(r,0)$. \\
Assume that $\calH' \neq \{0\}$. Then, applying the previous line of reasoning to $L(\calH')$ and $C(\calW)$ shows that
there is a non-trivial linear subspace $G'$ of $\K^r$ which is stable under $\widetilde{P}$ for all $P \in K(\calV)$
and such that $L(\calH')=\bigl\{M \in \Mat_{n-r,r}(\K) : \; G' \subset \Ker M\bigr\}$
and $C(\calW)=\bigl\{M \in \Mat_{r,n-r}(\K) : \; \im M \subset G'\bigr\}$.
Therefore, $G'=G=\K \times \{0\}$. Thus $L(\calW)=L(\calH')$, which leads to $\calW=\calH+\calH'$.
Let $L \in \Mat_{n-r,r-1}(\K)$ and $C \in \Mat_{1,n-r}(\K)$.
Then, $\calV$ contains a matrix of the form
$$M=\begin{bmatrix}
0 & [0]_{1 \times (r-1)} & C \\
[0]_{(r-1) \times 1} & [0]_{(r-1) \times (r-1)} & [0]_{(r-1) \times (n-r)} \\
[0]_{(n-r) \times 1} & L & N
\end{bmatrix} \quad \text{for some $N \in \Mat_{n-r}(\K)$.}$$
Now, remember that $\calV$ contains $J_r=I_r \oplus 0_{n-r}$. Applying identity \eqref{general} to $J_r$ and $M$ yields $N=0$.
It follows that $\calV$ satisfies the assumptions of Lemma \ref{lastlemma}, and hence either $\calV$ is equivalent to $\calR(1,r-1)$,
or $(n,r,\# \K)=(3,2,2)$ and $\calV$ is equivalent to $\calR(1,1)$ or to $\calJ_3(\K)$.

Assume finally that $q=r-1$ and $q>1$ (which discards the special case when $(n,r,\# \K)=(3,2,2)$). Then,
the above line of reasoning shows that we may assume that $\calW=\calH+\calH'$,
with $L(\calH')=\bigl\{M \in \Mat_{n-r,r}(\K) : \; \K^{r-1} \times \{0\} \subset \Ker M\bigr\}$
and $C(\calH)=\bigl\{M \in \Mat_{r,n-r}(\K) : \; \im M \subset \K^{r-1} \times \{0\}\bigr\}$.
As $(\Mat_{r-1}(\K) \vee \Mat_1(\K))^T$ is similar to $\Mat_1(\K) \vee \Mat_{r-1}(\K)$, we see that, for
$\calV':=\calV^T$, we now have $q'=1$, which brings us back to the previous situation.
In that case, we deduce that $\calV$ is equivalent to $\calR(1,r-1)^T=\calR(r-1,1)$.

This completes the proof of Theorem \ref{square}.

\subsection{The case $\dim \calV=nr-r+1+p-n$ for non-square matrices}\label{limitnonsquare}

Here, we assume that $n>p$.
If $q=0$, then $\calH=\{0\}$ and we are reduced to the situation of the Common Kernel Theorem.
Now, we assume that $q \geq 1$.
Again, we denote by $\calH'$ the linear subspace of $\calV$ consisting of its matrices $M$ for which
$K(M)=0$ and $C(M)=0$, with the block formats from Paragraph \ref{start}.

Assume that $L(\calH')=\{0\}$.
Then, $\calH'=\{0\}$ and $\dim \calV \leq p\,r$. However,
$n\,r+(p-n)-r+1-r\,p=(r-1)(n-p-1)\geq 0$, which leads to $r=1$, $n=p+1$ and $\dim \calV=p\,r$.
The Common Kernel Theorem applies to $\calV^T$, which shows that $\calV$ is equivalent to $\calR(r,0)$.
\begin{center}
From now on, we assume that $L(\calH') \neq \{0\}$.
\end{center}
If $\dim K(\calV)>r^2-r+1$, then Lemma \ref{geninverse} yields $L(\calW)=\{0\}$
(remember that $q>0$), and we deduce that $L(\calH')=\{0\}$, which contradicts our assumptions.
Therefore, $\dim K(\calV) \leq r^2-r+1$.
Remembering the inequalities
$$\dim C(\calH) \leq q\,(p-r) \quad \text{and} \quad \dim L(\calW) \leq (r-q)\,(n-r),$$
the rank theorem shows that
$$\dim \calV \leq r^2-r+1+\dim C(\calH)+\dim L(\calW) \leq n\,r+(p-n)\,q-r+1,$$
which in turn yields $q=1$. From there, we may use the same line of reasoning
as in Paragraph \ref{limitsquare}:
\begin{itemize}
\item One shows that $K(\calV)$ satisfies the assumptions of Lemma \ref{inversestab}, which entails that
$K(\calV)$ is similar to $\Mat_1(\K) \vee \Mat_{r-1}(\K)$ and $r>1$; no generality is then lost in assuming that
$K(\calV)=\Mat_1(\K) \vee \Mat_{r-1}(\K)$, which entails that $G=\K \times \{0\}$.
\item Then, one finds $C(\calH)=\bigl\{M \in \Mat_{r,p-r}(\K) : \; \im M \subset \K \times \{0\}\bigr\}$
and $L(\calW)=\bigl\{M \in \Mat_{n-r,r}(\K) : \; \K \times \{0\} \subset \Ker M\bigr\}$
\item Using $L(\calH')\neq \{0\}$ and $L(\calH') \subset L(\calW)$, together with similar dimensional arguments as above, one finds
$L(\calH')=\bigl\{M \in \Mat_{n-r,r}(\K) : \; \K \times \{0\} \subset \Ker M\bigr\}$.
\item One deduces that $\calW=\calH+\calH'$, and then one proves that $\calV$ satisfies all the assumptions of Lemma \ref{lastlemma}.
\end{itemize}
As the assumption $n>p$ discards the exceptional case in Lemma \ref{lastlemma}, one deduces that
$\calV$ is equivalent to $\calR(1,r-1)$. This completes the proof of Theorem \ref{rectangular}.

\section{The case of $\Mat_3(\F_2)$}\label{M3F2section}

Here, we seek to classify all the $5$-dimensional $\overline{2}$-subspaces of $\Mat_3(\F_2)$,
as stated in Theorem \ref{M3F2}. Let $\calV$ be such a subspace.
Notice that the proofs from Section \ref{reductionII} show that
either $\calV$ satisfies the assumptions of the Common Kernel Theorem,
or $\calV^T$ satisfies them, or $\calV$ is equivalent to $\calJ_3(\F_2)$. \\
Notice also that $\calJ_3(\F_2)$ and $\calJ_3(\F_2)^T$ are actually equivalent, and even similar:
this is easily seen by remarking that
$\calJ_3(\F_2)^T=\frak{sl}_3(\F_2) \cap T_3^+(\F_2)$ and by conjugating $\calJ_3(\F_2)^T$ with the permutation matrix
$\begin{bmatrix}
0 & 0 & 1 \\
0 & 1 & 0 \\
1 & 0 & 0
\end{bmatrix}$.
Thus, it suffices to prove the following proposition, which describes the exceptional cases in the Common Kernel Theorem:

\begin{prop}\label{noncomkerM3F2}
Let $\calW$ be a linear subspace of
$\Mat_{3,2}(\F_2)$ such that $\dim \calW=5$. Let
$\varphi : \calW \rightarrow \Mat_{3,1}(\F_2)$ be a linear map.
Assume that
$$\calV:=\Bigl\{\begin{bmatrix}
M & \varphi(M)
\end{bmatrix}\mid M \in \calW\Bigr\}.$$
is a $\overline{2}$-subspace of $\Mat_3(\F_2)$.
Then, either $\calV$ is equivalent to a subspace of $\calR(0,2)$ or it is equivalent to $\calJ_3(\F_2)$.
\end{prop}

Note that $\calW^\bot$ contains only one non-zero matrix $B$. Using a series of row and column operations,
we see that no generality is lost in assuming that $B=\begin{bmatrix}
0 & 1 \\
0 & 0 \\
0 & 0
\end{bmatrix}$ or
$B=\begin{bmatrix}
1 & 0 \\
0 & 1 \\
0 & 0
\end{bmatrix}$. We will deal with these two cases separately
(in the second one, we will show that $\calV$ is equivalent to a subspace of $\calR(0,2)$).

\subsection{The case $\rk B=1$}

We assume that $B=\begin{bmatrix}
0 & 1 \\
0 & 0 \\
0 & 0
\end{bmatrix}$, whence
$$\calW=\Biggl\{\begin{bmatrix}
a & 0 \\
b & d \\
c & e
\end{bmatrix}\mid  (a,b,c,d,e)\in \F_2^5\Biggr\}.$$
Since $\Mat_2(\F_2)$ is spanned by $\GL_2(\F_2)$, the line of reasoning from the proof of Proposition \ref{comstep1}
shows that the first entry of $\varphi(M)$ only depends on the entry of $M$ at the $(1,1)$-spot;
thus, either
$\forall M \in \calW, \; \varphi(M)=\begin{bmatrix}
0 \\
? \\
?
\end{bmatrix}$ or
$\forall M \in \calW, \; \varphi(M)=\begin{bmatrix}
m_{1,1} \\
? \\
?
\end{bmatrix}$. In the latter case, the column operation $C_3 \leftarrow C_3+C_1$ shows that we lose no generality in assuming that
every matrix of $\calV$ splits up as
$$\begin{bmatrix}
a & 0 & 0 \\
b & d & ? \\
c & e & ?
\end{bmatrix}.$$
Now, we write every matrix of $\calV$ as
$$M=\begin{bmatrix}
m_{1,1} & [0]_{1 \times 2} \\
[?]_{2 \times 1} & K(M)
\end{bmatrix} \quad \text{with $K(M)\in \Mat_2(\F_2)$.}$$
Then,
$$\calY:=\bigl\{K(M) \mid M \in \calV\; \text{such that}\; m_{1,1}=1\bigl\}$$
is an affine subspace of $\Mat_2(\K)$ with rank $\leq 1$ and translation vector space
$$Y=\bigl\{K(M) \mid M \in \calV\; \text{such that}\; m_{1,1}=0\bigl\}.$$
Note that, for all $(d,e)\in \K^2$, the vector space $Y$ contains a matrix of the form
$\begin{bmatrix}
d & ? \\
e & ?
\end{bmatrix}$, and hence $Y$ is inequivalent to $\calR(1,0)$ and $\dim \calY \geq 2$.

By Dieudonn\'e's theorem for affine spaces of singular matrices \cite{Dieudonne}, there are three possibilities for $\calY$:
\begin{itemize}
\item Either $\calY$ is equivalent to $\calR(1,0)$: this is impossible, as it would imply that $Y=\calY$ is equivalent to $\calR(1,0)$.
\item Or all the matrices of $\calY$ vanish at some common non-zero vector $z\in \K^2$; then,
all the matrices of $Y$ also vanish at $z$; one deduces a non-zero vector of $\K^3$ at which
all the matrices of $\calV$ vanish, and hence $\calV$ is equivalent to a subspace of $\calR(0,2)$.

\item Or $\calY$ is equivalent to the affine space of all matrices of the form
$\begin{bmatrix}
x & y \\
0 & x+1
\end{bmatrix}$ with $(x,y)\in \K^2$. Using row and column operations on the second and third rows and on the second and third columns
of the matrices of $\calV$, we see that no generality is lost in actually assuming that
$$\calY=\biggl\{\begin{bmatrix}
x & 0 \\
y & x+1
\end{bmatrix} \mid (x,y)\in \K^2\biggr\}.$$
Then, one deduces that
$$Y=\biggl\{\begin{bmatrix}
x & 0 \\
y & x
\end{bmatrix} \mid (x,y)\in \K^2\biggr\},$$
and one concludes that every matrix of $\calV$ has the form
$$M=\begin{bmatrix}
m_{1,1} & 0 & 0 \\
? & m_{2,2} & 0 \\
? & ? & m_{2,2}+m_{1,1}
\end{bmatrix}.$$
Thus, $\calV \subset \calJ_3(\F_2)$, and the equality of dimensions yields $\calV=\calJ_3(\F_2)$.
\end{itemize}

This completes our study of the case when $\rk B=1$.

\subsection{The case $\rk B=2$}

Here, we assume that $B=\begin{bmatrix}
1 & 0 \\
0 & 1 \\
0 & 0
\end{bmatrix}$, to the effect that
$$\calW=\Biggl\{\begin{bmatrix}
N \\
L
\end{bmatrix}\mid N \in \frak{sl}_2(\F_2), \; L \in \Mat_{1,2}(\F_2)\Biggr\}.$$
Noticing that $\frak{sl}_2(\F_2)$ is spanned by its non-singular matrices, we may use the arguments from the proof
of Proposition \ref{comstep1} to see that the third entry of $\varphi(M)$ depends only on the last row of $M$.
Then, we may write
$\varphi \begin{bmatrix}
N \\
L
\end{bmatrix}=\begin{bmatrix}
? \\
\delta(L)
\end{bmatrix}$ for some linear form $\delta : \Mat_{1,2}(\F_2) \rightarrow \F_2$.
Using the column operation $C_3 \leftarrow C_3+a\,C_1+b\,C_2$ for some well-chosen $(a,b)\in \F_2^2$, we may reduce the situation to the one
where every matrix of $\calV$ has entry zero at the $(3,3)$-spot.
It follows that:
\begin{itemize}
\item For every $N \in \frak{sl}_2(\F_2)$, there is a unique $\gamma(N) \in \Mat_{2,1}(\F_2)$ such that $\calV$
contains the matrix
$\begin{bmatrix}
N & \gamma(N)\\
0 & 0
\end{bmatrix}$;
\item For every $L \in \Mat_{1,2}(\F_2)$, there is a unique $\alpha(L) \in \Mat_{2,1}(\F_2)$ such that $\calV$ contains the matrix
$\begin{bmatrix}
[0]_{2 \times 2} & \alpha(L)\\
L & 0
\end{bmatrix}$.
\end{itemize}
Let $N \in \frak{sl}_2(\F_2)$ and $L \in \Mat_{1,2}(\F_2)$. Identity \eqref{comatrixformula} yields
$$L\,\widetilde{N}\,\bigl(\gamma(N)+\alpha(L)\bigr)=0.$$
Since $N \in \frak{sl}_2(\F_2)$, we actually have
$\widetilde{N}=N$, whence
\begin{equation}\label{lasteq}
\forall (N,L) \in \frak{sl}_2(\F_2) \times \Mat_{1,2}(\F_2), \quad LN\gamma(N)=LN\alpha(L).
\end{equation}
With $L$ fixed, notice that the left hand-side of \eqref{lasteq} is a quadratic form of $N$, and the right hand-side
is a linear form. We deduce that
$$\forall (M,N,L) \in \frak{sl}_2(\F_2)^2 \times \Mat_{1,2}(\F_2), \; L\,(M\gamma(N)+N\gamma(M))=0,$$
and hence
\begin{equation}\label{reallasteq}
\forall (M,N) \in \frak{sl}_2(\F_2)^2, \; M\gamma(N)+N\gamma(M)=0.
\end{equation}
Notice that $\Ker \gamma \neq \{0\}$, by the rank theorem.

Assume that $\Ker \gamma$ contains only singular matrices. Then, $\dim \Ker \gamma=1$ as the three
rank $1$ matrices of $\frak{sl}_2(\F_2)$ are linearly independent. It follows that $\gamma$ maps $\frak{sl}_2(\F_2)$ onto $\Mat_{2,1}(\F_2)$.
However, choosing $M_0 \in \Ker \gamma \setminus \{0\}$, we find that
$\forall N \in \frak{sl}_2(\F_2), \; M_0\,\gamma(N)=0$, and hence $M_0=0$, a contradiction.

We deduce that $\Ker \gamma$ contains a non-singular matrix $M_0$.
Then, \eqref{reallasteq} entails $M_0\,\gamma(N)=0$ for every $N \in \frak{sl}_2(\F_2)$, which yields $\gamma=0$.

It follows that, given an arbitrary $L \in \Mat_{1,2}(\F_2) \setminus \{0\}$, one has $LN\alpha(L)=0$ for every $N \in \frak{sl}_2(\F_2)$;
checking that $\Vect(L\frak{sl}_2(\F_2))=\Mat_{1,2}(\F_2)$, we deduce that $\alpha(L)=0$.

We conclude that $\alpha=0$ and $\gamma=0$, which yields $\calV \subset \calR(0,2)$.
This finishes the proof of Proposition \ref{noncomkerM3F2}, and Theorem \ref{M3F2} ensues.

\subsection{A final remark}

Using the previous results, the reader will easily prove the following generalization of Proposition \ref{noncomkerM3F2},
which fully describes the exceptional case in the Common Kernel Theorem:

\begin{prop}
Let $\calW$ be a linear subspace of
$\Mat_{3,2}(\F_2)$ such that $\dim \calW=5$. Let $p \geq 3$, and
$\varphi : \calW \rightarrow \Mat_{3,p-2}(\F_2)$ be a linear map.
Assume that
$$\calV=\Bigl\{\begin{bmatrix}
M & \varphi(M)
\end{bmatrix}\mid M \in \calW\Bigr\}$$
is a $\overline{2}$-subspace of $\Mat_{3,p}(\F_2)$. Then, either $\calV$ is
equivalent to a subspace of $\calR(0,2)$, or it is equivalent to the vector space
$$\Bigl\{\begin{bmatrix}
M & [0]_{3 \times (p-3)}
\end{bmatrix} \mid M \in \calJ_3(\F_2)\Bigr\} \subset \Mat_{3,p}(\F_2).$$
\end{prop}

\section{The second classification theorem}

This last section consists in a proof of Theorem \ref{class2theorem}.
In this prospect, we shall need a classical lemma of Flanders on $\overline{r}$-spaces of matrices,
along with inverse transitivity results of the same flavor as Lemma \ref{geninverse}.
The first two paragraphs are devoted to such lemmas.
Then, we shall proof Theorem \ref{class2theorem} by relying upon the Atkinson-Lloyd-Beasley classification theorems:
the proof will be based upon block-decompositions, with a method
that has similarities with the one of Section \ref{start} but is closer to Flanders's original line of reasoning.

Key to the proof is the reduction to the following situation:

\begin{Def}
Let $n$ and $p$ be positive integers, and $r$ be a positive integer with $r<\min(n,p)$.
A linear subspace $\calV$ of $\Mat_{n,p}(\K)$ has \textbf{property $(\calP_r)$} when there is a linear subspace
$\calW$ of $\Mat_r(\K)$ such that:
\begin{enumerate}[(i)]
\item $\calW$ contains $I_r$.
\item $\dim \calW \geq r^2-2r+4$ (note that this implies $r \geq 2$).
\item For every $N \in \calW$, the space $\calV$ contains $\begin{bmatrix}
N & [0]_{r \times (p-r)} \\
[0]_{(n-r) \times r} & [0]_{(n-r) \times (p-r)}
\end{bmatrix}$.
\end{enumerate}
\end{Def}

Here is our key result, which will be proved in Paragraph \ref{superRreducedsection}:

\begin{prop}\label{superRreduced}
Let $n$ and $p$ be positive integers, and $r$ be a positive integer with $r<\min(n,p)$.
Let $\calV$ be an $\overline{r}$-subspace of $\Mat_{n,p}(\K)$ with property $(\calP_r)$.
Then, $\calV$ is equivalent to a subspace of one of the spaces $\calR(i,r-i)$ or
$\calR(r-i,i)$, for some $i \in \{0,1,2\}$.
\end{prop}

\subsection{Basic lemmas}

The following two lemmas are standard tools in the analysis of
$\overline{r}$-spaces of matrices:

\begin{lemme}[Flanders, \cite{Flanders} Lemma 1]\label{Flanderslemma}
Let $\calV$ be an $\overline{r}$-subspace of $\Mat_{n,p}(\K)$ which contains $I_r \oplus 0$.
Assume that $\# \K>r$.
Then, every matrix of $\calV$ splits up as
$$M=\begin{bmatrix}
[?]_{r \times r} & C(M) \\
L(M) & [0]_{(n-r) \times (p-r)}
\end{bmatrix},$$
where $L(M)$ and $C(M)$ are $(n-r) \times r$ and $r \times (p-r)$ matrices, respectively, that satisfy $L(M)C(M)=0$.
\end{lemme}

\begin{lemme}\label{ranksplittinglemma}
Let $\calV$ be an $\overline{r}$-subspace of $\Mat_{n,p}(\K)$, and let $(s,t) \in \lcro 1,n-1\rcro \times \lcro 1,p-1\rcro$.
Assume that every matrix of $\calV$ splits up as
$$M=\begin{bmatrix}
[?]_{s \times t} & C(M) \\
L(M) & [0]_{(n-s) \times (p-t)}
\end{bmatrix}.$$
If $\# \K>r$, then $\rk L(\calV)+\rk C(\calV) \leq r$.
\end{lemme}

Lemma \ref{ranksplittinglemma} is proved in the course of the proof of Lemma 4 of \cite{AtkLloyd}. \\
We finish with a lemma that will be used in the proof of Proposition \ref{superRreduced}:

\begin{lemme}\label{uniquestablesubspace}
Let $\calW$ be a linear subspace of $\Mat_r(\K)$ such that $\dim \calW \geq r^2-2r+4$
(to the effect that $r \geq 2$).
Then, there is at most one non-trivial subspace of $\K^r$ that is stable under all the elements of $\calW$.
\end{lemme}

\begin{proof}
Let $H_1$ and $H_2$ be two distinct non-trivial linear subspaces of $\K^r$ that are stable under all the elements of $\calW$.
Denote by $p$ and $q$ their respective dimensions.
Assume first that $H_1 \cap H_2=\{0\}$. Then, $\calW$ is equivalent to a linear subspace
of $\bigl(\Mat_p(\K) \oplus \Mat_q(\K)\bigr)\vee \Mat_{r-p-q}(\K)$, so that
$$\codim_{\Mat_r(\K)} \calW \geq (r-p)p+(r-q)q  \geq 2(r-1).$$
Here, we have used the fact that the function $t \mapsto (r-t)t$ is increasing on $\bigl[0,\frac{r}{2}\bigr]$ and symmetric around 
$\frac{r}{2}$, to the effect that its minimal value on $[1,r-1]$ is $r-1$. 

If we now assume that $H_1 \cap H_2 \neq \{0\}$, we see that no generality is lost in assuming that $H_1 \subsetneq H_2$
(as $H_1 \cap H_2$ is a proper subspace of either $H_1$ or $H_2$, and it is stable under all the elements of $\calW$).
In that case, we see that $\calW$ is equivalent to a subspace of $\Mat_p(\K) \vee \Mat_{q-p}(\K) \vee \Mat_{r-q}(\K)$, and hence
$$\codim_{\Mat_r(\K)} \calW \geq (r-p)p+(r-q)(q-p) \geq (r-p)p+(r-p-1) \geq 2r-3,$$
where we have used the fact that the function $t \mapsto (r-t)(t-p)$ is increasing on
$\bigl[p,\frac{r+p}{2}\bigr]$ and symmetric around $\frac{r+p}{2}$ (to the effect that its minimal value on
$\lcro p+1,r-1\rcro$ is obtained for $t=p+1$), while the function
$s \mapsto (r-s)s+(r-s-1)=(r-s)(s+1)-1$ is increasing on $\bigl[0,\frac{r-1}{2}\bigr]$ and symmetric around $\frac{r-1}{2}$
(to the effect that its minimal value on $\lcro 1,r-2\rcro$ is obtained for $s=1$).

In any case, we have contradicted the assumption that $\dim \calW \geq r^2-2r+4$.
\end{proof}

\subsection{Additional inverse transitivity lemma}

\begin{lemme}\label{comatrix1lemma}
Let $\calW$ be an affine subspace of $\Mat_1(\K) \vee \Mat_{r-1}(\K)$ such that
$\dim \calW \geq r^2-2r+4$. Assume furthermore that $\calW$ contains a non-singular matrix.
Then,
$$\forall x \in \K^r \setminus (\K \times \{0\}), \; \Vect \bigl\{P^{-1}x \mid P \in \GL_r(\K) \cap \calW\bigr\}=\K^r.$$
\end{lemme}

\begin{proof}
Assume on the contrary that there is a linear hyperplane $H$ of $\K^r$ and a vector $x \in \K^r \setminus (\K \times \{0\})$
such that $P^{-1} x\in H$ for all $P \in \GL_r(\K) \cap \calW$.
Denote by $(e_1,\dots,e_r)$ the canonical basis of $\K^r$.
As the situation is unchanged by multiplying $\calW$ on the left and on the right with non-singular matrices which belong to
$\Mat_1(\K) \vee \Mat_{r-1}(\K)$, we deduce that no generality is lost in assuming that $x=e_2$ and that either
$H=\Vect(e_2,\dots,e_r)$ or $H=\Vect(e_1,\dots,e_{r-1})$ (this corresponds to the situations where $e_1 \not\in H$ and $e_1 \in H$, respectively).
In the first case (respectively, in the second one), this means that the entry of $P^{-1}$ at the $(1,2)$-spot
(respectively, at the $(r,2)$-spot) is zero for all $P \in \calW \cap \GL_r(\K)$.

As $\calW$ contains an invertible matrix and is included in $\Mat_1(\K) \vee \Mat_{r-1}(\K)$, the set of
matrices $M \in \calW$ such that $m_{1,1}=1$ is an affine subspace $\calH$ of dimension at least $r^2-2r+3$.
For $M \in \calH$, we write
$$M=\begin{bmatrix}
1 & L(M) \\
[0]_{(r-1) \times 1} & Q(M)
\end{bmatrix} \quad \text{with $L(M) \in \Mat_{1,r-1}(\K)$ and $Q(M) \in \Mat_{r-1}(\K)$.}$$
Then, the rank theorem shows that $Q(\calH)$ is an affine subspace of $\Mat_{r-1}(\K)$ and
$$\dim Q(\calH)>(r-1)^2-(r-1)+1.$$
Applying Lemma \ref{inversestab}, we find some $M \in \calH$ for which $Q(M)$ is invertible and $Q(M)^{-1}$ has a non-zero entry at the
$(r-1,1)$-spot, which discards the second case above. Therefore, the inverse of every invertible matrix of $\calW$ has entry $0$ at the $(1,2)$-spot.

Finally, by the rank theorem, one finds that the translation vector space of $\calW$ contains a non-zero matrix of the form
$$M_0=\begin{bmatrix}
0 & L_0 \\
[0]_{(r-1) \times 1} & [0]_{(r-1) \times (r-1)}
\end{bmatrix} \quad \text{with $L_0 \in \Mat_{1,r-1}(\K)$.}$$
Let $M \in \calH$ be such that $Q(M)$ is invertible. For every $\lambda \in \K$, we find that $\lambda M_0+M$ is invertible and
$$(\lambda M_0+M)^{-1}=\begin{bmatrix}
1 & -\bigl(\lambda L_0+L(M)\bigr) Q(M)^{-1} \\
[0]_{(r-1) \times 1} & Q(M)^{-1}
\end{bmatrix}.$$
Thus, the first entry of the row matrix $\bigl(\lambda L_0+L(M)\bigr) Q(M)^{-1}$ must be zero, and, varying $\lambda$, one deduces that
the first entry of $L_0 Q(M)^{-1}$ is zero. In other words, with $C_0:=\begin{bmatrix}
1 & 0 & \cdots & 0
\end{bmatrix}^T$, we have $L_0 R^{-1} C_0=0$ for every non-singular matrix $R$ of $Q(\calH)$. Then, Lemma \ref{geninverse}
yields a contradiction.
\end{proof}

\subsection{The structure of spaces with property $(\calP_r)$}\label{superRreducedsection}

Here, we prove Proposition \ref{superRreduced}.
Let $\calV$ be an $\overline{r}$-subspace of $\Mat_{n,p}(\K)$, and assume that there is a subspace $\calW$ of $\Mat_r(\K)$
satisfying conditions (i) to (iii) for $\calV$ in the definition of property $(\calP_r)$.
In particular, $\calV$ contains $I_r \oplus 0$: using Lemma \ref{Flanderslemma}, we see that every matrix of
$\calV$ splits up as
$$M=\begin{bmatrix}
K(M) & C(M) \\
L(M) & [0]_{(n-r) \times (p-r)}
\end{bmatrix},$$
where $K(M)$, $L(M)$ and $C(M)$ are $r \times r$, $(n-r) \times r$ and $r \times (p-r)$ matrices, respectively.
If $L$ (respectively, $C$) vanishes everywhere on $\calV$, we have $\calV \subset \calR(r,0)$ (respectively, $\calV \subset \calR(0,r)$)
and we are done. For the rest of the proof, we shall assume that neither $L$ nor $C$ vanish everywhere on $\calV$.

Note that the space $\calV$ is spanned by its matrices $M$ satisfying $L(M) \neq 0$ and $C(M) \neq 0$.
Indeed, as $r \geq 2$, $\K$ has at least three elements, and hence $\calV$ is not included in the union of three of its proper linear subspaces; in particular, $\calV \setminus (\Ker L \cup \Ker C)$ is not included in a proper linear subspace of $\calV$.

Let $M_0 \in \calV$ be such that $L(M_0)\neq 0$ and $C(M_0) \neq 0$, and write $N_0=K(M_0)$.
Identity \eqref{comatrixformula} yields that
$$\forall N \in \calW, \; L(M_0)\widetilde{(N_0+N)}C(M_0)=0.$$
Fixing $N \in \calW$, we see that the entries of $L(M_0)\widetilde{(tN_0+sN)}C(M_0)$
are homogeneous polynomials of degree $r-1$ in the variable $(t,s)$, and they vanish at every
point of the projective line $\mathbb{P}(\K^2)$ with the possible exception of the point with homogeneous coordinates $[1,0]$.
As $\mathbb{P}(\K^2)$ has more than $r+1$ points, one deduces that all those polynomials are zero, and hence
\begin{equation}\label{superrefinedcomatrix}
\forall N \in \K N_0 + \calW, \; L(M_0)\widetilde{N}C(M_0)=0.
\end{equation}
As $L(M_0) \neq 0$ and $C(M_0) \neq 0$, one deduces that there is a non-zero vector $x \in \K^r$
such that $\Vect \{\widetilde{N}x \mid N \in \K N_0+\calW\}$ is a proper subspace of $\K^r$.

Thus, we find matrices $Q_1$ and $Q_2$ of $\GL_r(\K)$ such that, for every
$N \in \K N_0+\calW$, the upper-left $(r-1) \times (r-1)$-submatrix of $Q_1NQ_2$ is singular.
Splitting $N=Q_1^{-1} \begin{bmatrix}
H(N) & [?]_{(r-1) \times 1} \\
[?]_{1 \times (r-1)} & ?
\end{bmatrix} Q_2^{-1}$ for all such $N$, we deduce that $H(\K N_0+\calW)$ is an $\overline{r-2}$-subspace of $\Mat_{r-1}(\K)$.
However, the rank theorem yields that
$$\dim H(\K N_0+\calW) \geq \dim \calW-2r+1 \geq (r-1)(r-2)-(r-2)+1.$$
Using Theorem \ref{square}, we deduce that:
\begin{enumerate}[(i)]
\item Either $H(\K N_0+\calW)$ is equivalent to a subspace of $\calR(0,r-2)$;
\item Or $H(\K N_0+\calW)$ is equivalent to a subspace of $\calR(r-2,0)$;
\item Or $r > 3$ and $H(\K N_0+\calW)$ is equivalent to a subspace of $\calR(1,r-3)$;
\item Or $r > 3$ and $H(\K N_0+\calW)$ is equivalent to a subspace of $\calR(r-3,1)$.
\end{enumerate}
Replacing $\calV$ with its transpose if necessary, we see that only cases (i) and (iii) need to be tackled:
indeed, one notes that $\calV^T \subset \Mat_{p,n}(\K)$ has property $(\calP_r)$ and
that if the desired conclusion holds for $\calV^T$, then it holds for $\calV$ as well.

In any of cases (i) and (iii) above, we obtain proper $i$-dimensional subspaces $P$ and $P'$ of $\K^r$
such that every matrix of $\K N_0+\calW$ maps $P$ into $P'$, more precisely, we can take $i=1$ in case (i), and $i=2$ in case (iii).
As $\calV$ contains $I_r \oplus 0$, one notes that $\K N_0+\calW$ contains $I_r$, to the effect that $P'=P$.
Therefore, we have found that $P$ is a non-trivial linear subspace of $\K^r$ which is stable under all the matrices of
$\K N_0+\calW$. By Lemma \ref{uniquestablesubspace}, it is the sole non-trivial linear subspace which is stable under all the matrices
of $\calW$. Thus, we may choose a matrix $Q \in \GL_r(\K)$ which maps $P$ onto $\K^i \times \{0\}$
(and make such a choice with no regard to $M_0$)
and replace $\calV$ with $(Q \oplus I_{n-r}) \calV (Q \oplus I_{p-r})^{-1}$ to reduce the situation further to the point where
one of the following situations holds: \\

\noindent \textbf{Case 1.} Either $K(M) \in \Mat_1(\K) \vee \Mat_{r-1}(\K)$
for every $M \in \calV$ satisfying $L(M) \neq 0$ and $C(M) \neq 0$.
As those matrices span $\calV$, one deduces that $K(\calV) \subset \Mat_1(\K) \vee \Mat_{r-1}(\K)$.
\vskip 2mm
\noindent
\textbf{Case 2.} Or $r>3$ and $K(M) \in \Mat_2(\K) \vee \Mat_{r-2}(\K)$
for every $M \in \calV$ satisfying $L(M) \neq 0$ and $C(M) \neq 0$.
Again, this implies $K(\calV) \subset \Mat_2(\K) \vee \Mat_{r-2}(\K)$.

\vskip 2mm
Now, let us discuss those two cases separately. Assume that Case 1 holds.
Again, let $M_0 \in \calV$ be such that $L(M_0) \neq 0$ and $C(M_0) \neq 0$.
Using \eqref{superrefinedcomatrix}, one finds
$$\forall N \in \calW \cap \GL_r(\K), \; L(M_0)N^{-1}C(M_0)=0.$$
Using Lemma \ref{comatrix1lemma}, one deduces that
$$C(M_0)=\begin{bmatrix}
R_1 \\
[0]_{(r-1) \times (p-r)}
\end{bmatrix} \quad \text{for some $R_1\in \Mat_{1,p-r}(\K) \setminus \{0\}$.}$$
Denoting by $S_1$ the first column of $L(M_0)$,
identity $L(M_0)C(M_0)=0$ from Lemma \ref{Flanderslemma} reads $S_1R_1=0$, and hence $S_1=0$.
Thus, we have shown that every matrix $M \in \calV$ such that $L(M) \neq 0$ and $C(M) \neq 0$ has the form
$$\begin{bmatrix}
? & [?]_{1 \times (r-1)} & [?]_{1 \times (p-r)} \\
[0]_{(n-1) \times 1} & [?]_{(n-1)\times (r-1)} & [0]_{(n-1) \times (p-r)}
\end{bmatrix}.$$
As those matrices span $\calV$, the result holds for every matrix of $\calV$.
Permuting columns shows that $\calV$ is equivalent to a subspace of $\calR(1,r-1)$.

\vskip 2mm
We complete the proof by examining Case 2.
As $\dim K(\calV) \geq \dim \calW \geq r^2-2r+4$ and $K(\calV) \subset \Mat_2(\K) \vee \Mat_{r-2}(\K)$, we find
$K(\calV)=\calW=\Mat_2(\K) \vee \Mat_{r-2}(\K)$.
Again, let us fix $M_0 \in \calV$ such that $L(M_0) \neq 0$ and $C(M_0) \neq 0$.
Identity \eqref{superrefinedcomatrix} yields:
$$\forall N \in \calW \cap \GL_{r-1}(\K), \; L(M_0) N^{-1} C(M_0)=0.$$
With Lemma \ref{comatrix2lemma}, one deduces that $\im C(M_0) \subset  \K^2 \times \{0\}$.
Noting that for the permutation matrix $P$ associated with $i \mapsto r+1-i$, one has $\calW=P \calW^T P^{-1}$, one
deduces, with the same line of reasoning, that $\im L(M_0)^T \subset \{0\} \times \K^{r-2}$.
It follows that $M_0$ has the form
$$M_0=\begin{bmatrix}
[?]_{2 \times 2} & [?]_{2 \times (r-2)} & [?]_{2 \times (p-r)} \\
[0]_{(n-2) \times 2} & [?]_{(n-2) \times (r-2)} & [0]_{(n-2) \times (p-r)}
\end{bmatrix}.$$
As such matrices span $\calV$, we may use a permutation of columns to find that $\calV$
is equivalent to a linear subspace of $\calR(2,r-2)$.

This completes the proof of Proposition \ref{superRreduced}.

\subsection{Proof of the second classification theorem}

Here, we prove Theorem \ref{class2theorem}.
Let $n$, $p$ and $r$ be positive integers such that $n \geq p >r \geq 2$.
Let $\calV$ be an $\overline{r}$-subspace of $\Mat_{n,p}(\K)$ with $\dim \calV \geq nr-2r+4+2(p-n)$.

We can also assume that $\calV$ contains a rank $r$ matrix, for if it does not, then
$\calV$ is an $\overline{r-1}$ subspace of $\Mat_{n,p}(\K)$ of dimension greater than or equal to
$nr-2r+4+2(p-n) \geq n(r-1)-(r-1)+(p-n)+2$, and then Theorems \ref{square} and \ref{rectangular}
yield that $\calV$ is equivalent to a subspace of $\calR(0,r-1)$ or $\calR(r-1,0)$.

In the rest of the proof, we shall assume that $\calV$ contains a rank $r$ matrix.
Then, no generality is lost in assuming that $\calV$ contains $I_r \oplus 0$. Lemma \ref{Flanderslemma}
entails that every matrix $M \in \calV$ splits up as
$$M=\begin{bmatrix}
K(M) & C(M) \\
L(M) & [0]_{(n-r) \times (p-r)}
\end{bmatrix}$$
where $K(M)$, $L(M)$ and $C(M)$ are $r \times r$, $(n-r) \times r$ and $r \times (p-r)$ matrices, respectively.
\vskip 2mm
Formula \eqref{comatrixformula} yields
\begin{equation}
\label{refinedcomatrixformula}
\forall M \in \calV, \; L(M) \widetilde{K(M)} C(M)=0.
\end{equation}
Let us define $\calV'$ as the linear subspace of $\calV$ consisting of its matrices $M$ satisfying $C(M)=0$, and
$\calV''$ as the linear subspace of $\calV'$ consisting of its matrices $M$ satisfying $L(M)=0$.
In other words, $\calV''$ is the subspace of all matrices of $\calV$ which have the form
$$M=\begin{bmatrix}
K(M) & [0]_{r \times (p-r)} \\
[0]_{(n-r) \times r} & [0]_{(n-r) \times (p-r)}
\end{bmatrix}.$$
Polarizing the quadratic formula $\forall M \in \calV, \; L(M)C(M)=0$ yields
\begin{equation}\label{orthoequation}
\forall M \in \calV, \; \forall N \in \calV', \; L(N) C(M)=0.
\end{equation}
Thus, setting
$$G:=\sum_{M \in \calV} \im C(M) \quad \text{and} \quad q:=\dim G,$$
one deduces from \eqref{orthoequation} that
$$\dim L(\calV') \leq (n-r) \times (r-q) \quad \text{and} \quad \dim C(\calV) \leq q(p-r).$$
Setting
$$\calW:=K(\calV'') \subset \Mat_r(\K),$$
one deduces from the rank theorem that
$$\dim \calW=\dim \calV-\dim L(\calV')-\dim C(\calV) \geq r^2-2r+4+(q-2)(n-p).$$
Notice also that $I_r \in \calW$.
If $q \geq 2$ or $n=p$, then $\calV$ has property $(\calP_r)$, and the conclusion ensues by using Proposition \ref{superRreduced}.
If $q=0$, then $C(\calV)=\{0\}$ and hence $\calV \subset \calR(0,r)$.

Now, we assume that $q=1$ and $n>p$. Then, $G$ is a $1$-dimensional subspace of $\K^r$; choosing $Q \in \GL_r(\K)$ such that
$QG=\K \times \{0\}$ and replacing $\calV$ with $(Q \oplus 0_{n-r}) \calV (Q^{-1} \oplus 0_{p-r})$, we see that no generality is lost
in assuming that $G=\K \times \{0\}$. Then, every matrix of $\calV$ splits up as
$$M=\begin{bmatrix}
[?]_{1 \times r} & R(M) \\
S(M) & [0]_{(n-1) \times (p-r)}
\end{bmatrix}$$
where $R(M) \in \Mat_{1,p-r}(\K)$ and $S(M) \in \Mat_{n-1,r}(\K)$.
As $q=1$, we have $\rk R(\calV)=1$, and hence Lemma \ref{ranksplittinglemma} shows that $\rk S(\calV) \leq r-1$.
However, the rank theorem yields
$$\dim S(\calV) \geq nr-2r+4+2(p-n)-p > (n-1)(r-1)-(r-1)+1+(r-(n-1)),$$
the last inequality stemming from $p>r$.
As $n-1\geq p>r$, Theorem \ref{rectangular} yields a non-zero vector $x$ of $\K^r$ such that
every matrix of $\calS(\calV)$ vanishes at $x$.
However, as $I_r \oplus 0$ belongs to $\calV$, we see that $\calS(\calV)$ contains
$\begin{bmatrix}
[0]_{(r-1) \times 1} & I_{r-1} \\
[0]_{(n-r) \times 1} & [0]_{(n-r) \times (r-1)}
\end{bmatrix}$.
This shows that $x \in \K \times \{0\}$, which, in turn, yields that every matrix of $\calV$ has the form
$$\begin{bmatrix}
? & [?]_{1 \times (r-1)} & [?]_{1 \times (p-r)} \\
[0]_{(n-1) \times 1} & [?]_{(n-1) \times (r-1)} & [0]_{(n-1) \times (p-r)}
\end{bmatrix}.$$
Permuting columns, one concludes that $\calV$ is equivalent to a linear subspace of $\calR(1,r-1)$.

This completes the proof of Theorem \ref{class2theorem}.

\subsection{Final comments}

A key point in the above proof is the way we use the Atkinson-Lloyd classification theorem
to recover crucial information on the structure of $K(\calV'')$. This suggests an inductive strategy to obtain
classification theorems for smaller dimensions, at least for square matrices.
If, for some positive integer $i$, one has access to a classification theorem for
 $\overline{r}$-subspaces for dimensions that are greater than or equal to $nr-i(r-i)$, for all possible values of $r$, then those
 results help us understand the structure of the $K(\calV'')$ space, and one can possibly
use this insight to recover the structure of $\overline{r}$-spaces of square matrices with dimension greater than or equal to
$nr-(i+1)(r-(i+1))$.

\end{document}